\documentclass[12pt]{amsart}

\usepackage{amsmath}
\usepackage{amssymb}
\usepackage{amsthm}
\usepackage[latin1]{inputenc}
\usepackage{eurosym}
\usepackage[dvips]{graphics}
\usepackage{graphicx}
\usepackage{epsfig}
\usepackage{hyperref}
\usepackage{dsfont}
\usepackage{caption}
\usepackage{subcaption}
\usepackage{epstopdf}
\usepackage{color}

\usepackage{ifthen}

\makeindex

\newcommand{\argmin}{\operatorname{argmin}}

\newcommand{\Rr}{{\mathbb{R}}}

\newcommand{\Nn}{{\mathbb{N}}}

\newcommand{\bi}{{\bold{i}}}

\newcommand{\Mm}{{\mathcal{M}}}

\newcommand{\lsb}{\left[}
\newcommand{\rsb}{\right]}
\newcommand{\lb}{\left(}
\newcommand{\rb}{\right)}

\newcommand{\Ss}{\mathcal{S}}
 
\newcommand{\re}{\mathbb{R}}

\newtheorem{pro}{Proposition}

\newtheorem{hyp}{Assumption}
\newtheorem{rem}{Remark}

\theoremstyle{definition}

\begin{document}

\title[Monotone numerical methods for finite-state MFG]{Monotone numerical methods for finite-state mean-field games}

\author{Diogo A. Gomes}
\address[D. A. Gomes]{
        King Abdullah University of Science and Technology (KAUST), CEMSE Division, Thuwal 23955-6900. Saudi Arabia, and  
        KAUST SRI, Uncertainty Quantification Center in Computational Science and Engineering.}
\email{diogo.gomes@kaust.edu.sa}
\author{Jo\~ao Sa\'ude}
\address[J. Saude]{
        Carnegie Mellon University, Electrical and Computer Engineering department. 5000 Forbes Avenue Pittsburgh, PA 15213-3890 USA.}
\email{jsaude@andrew.cmu.edu}


\thanks{
        D. Gomes was partially supported by KAUST baseline and start-up funds and 
KAUST SRI, Center for Uncertainty Quantification in Computational Science and Engineering.  
        J. Sa\'ude was partially supported by FCT/Portugal through the CMU-Portugal Program.
}
\date{\today}
\begin{abstract}
        Here, we develop numerical methods for finite-state mean-field games (MFGs) that satisfy a monotonicity condition. 
    MFGs are determined by a system of differential equations with initial and terminal boundary conditions. 
    These non-standard conditions are the main difficulty in the numerical approximation of solutions. 
    Using the monotonicity condition, we build a flow that is a contraction and whose fixed points solve the MFG, both for stationary and time-dependent problems. We illustrate our methods in a MFG modeling the paradigm-shift problem. 
\end{abstract}
\maketitle
\section{Introduction} 
\label{sec:introduction}
        The mean-field game (MFG) framework \cite{Caines2, Caines1, ll1, ll2} models systems with many rational players (also see the surveys  \cite{GPV} and \cite{GS}). 
        In finite-state MFGs, players switch between a finite number of states, see \cite{GMS} for discrete-time and \cite{BHK, FG, GMS2, Gueant2}, and \cite{Gueant1} for continuous-time problems.
        Finite-state MFGs have applications in socio-economic problems, 
    for example, in paradigm-shift and consumer choice models \cite{BesancenotDogguy, GVW-dual,Gomes:2014kq}, and arise
     in the approximation of continuous state MFGs
     \cite{achdou2013finite, DY, AFG}.     
            
   Finite-state MFGs comprise systems of ordinary differential equations with initial-terminal boundary conditions. 
        Because of these conditions, the numerical computation of solutions is challenging. 
    Often, MFGs satisfy a monotonicity condition that was introduced in \cite{ll1}, and \cite{ll2} to study the uniqueness of solutions. 
   Besides the uniqueness of solution, monotonicity implies the long-time convergence of MFGs, see  \cite{FG} and \cite{GMS2} for finite-state models and \cite{CLLP} and \cite{cllp13} for continuous-state models. 
   Moreover, monotonicity conditions were used in \cite{FG2} to 
   prove the existence of solutions to MFGs and in \cite{FG}
   to construct numerical methods for stationary MFGs.   
        Here, we consider MFGs that satisfy a monotonicity condition and develop a numerical method to compute their solutions. For stationary 
        problems, our method is a modification of the one in \cite{FG}. The main contribution of this paper concerns the handling of the initial-terminal boundary conditions, where the methods from \cite{FG} cannot be applied directly.

        We consider MFGs where each of the players can be at a state in $I_d=\{1,\ldots,d\}$, $d\in \Nn$, $d>1$,  the players' state space.
        Let $\Ss^d=\{\theta\in (\Rr_0^+)^d : \sum_{i=1}^d \theta^i=1\}$ be the probability simplex in $I_d$. 
        For a time horizon, $T>0$, the macroscopic description of the game is determined by a path $\theta: [0,T] \to \Ss^d$ that gives the probability distribution of the players in $I_d$. 
        All players seek to minimize an identical cost. 
        Each coordinate, $u^i(t)$, of the value function,  $u: [0,T] \to \Rr^d$,  is the minimum cost for a typical player at state $i\in I_d$ at time $0\leq t\leq T$. 
        Finally, at the initial time, the players are distributed according to the probability vector $\theta_0\in \Ss^d$ and, at the terminal time, are charged a cost $u_T\in \Rr^d$ that depends on their state. 

       In the framework presented in \cite{GMS}, finite-state MFGs have   a Hamiltonian, $h:\Rr^d\times \Ss^d\times I_d\to \Rr$, and a switching rate, $\alpha^*_i:\Rr^d\times \Ss^d\times I_d\to \Rr_0^+$, given by
        \begin{equation}
                \label{alfstar2}
                \alpha^*_j=\frac{\partial h(\Delta_iz,\theta,i) }{\partial z^j}, 
        \end{equation}
        where $\Delta_i:\Rr^d\to \Rr^d$ is the difference operator
        \[
                (\Delta_i u)^j=u^j-u^i.
        \]
        We suppose that $h$ and $\alpha^*$ satisfy the assumptions discussed in Section \ref{assp}. Given the Hamiltonian and the switching rate, 
        we assemble the system of differential equations:
        \begin{equation} \label{MFM:cont_dyn}
        \begin{cases}
        u_t^i = -h(\Delta_i u,\theta,i)\\
        \theta_t^i=\sum_j \theta^j \alpha^*_i(\Delta_j u, \theta, j), 
        \end{cases}
        \end{equation}
        which, together with initial-terminal data
        \begin{equation}
        \label{itc}
        \theta(0) = \bar \theta_0 \text{ and } u(T)=\bar u_T, 
        \end{equation}
        with $\bar \theta_0\in \Ss^d$ and $\bar u_T\in \Rr^d$, determines the MFG.

Solving \eqref{MFM:cont_dyn} under
the non-standard boundary condition \eqref{itc} 
is a fundamental issue in time-dependent MFGs.
There are several ways to address this issue, but 
prior approaches are not completely satisfactory. 
First, we can solve \eqref{MFM:cont_dyn} using initial conditions
$\theta(0)=\bar \theta_0$ and $u(0)=u_0$ and then solve for $u_0$ 
such that $u(T)=\bar u_T$. 
However, this requires solving \eqref{MFM:cont_dyn} multiple times, which is computationally expensive. A more fundamental difficulty 
arises in the numerical approximation 
of continuous-state MFGs by finite-state MFGs. There, 
the Hamilton-Jacobi equation is a backward parabolic equation 
whose initial-value problem is ill-posed. Thus, a possible way 
to solve \eqref{MFM:cont_dyn} is to use a Newton-like iteration. This idea
was developed in \cite{achdou2013finite, MR2928376} and used to solve
a finite-difference scheme for a continuous-state MFG. However, Newton's method involves inverting large matrices and, thus, it is convenient to have 
algorithms that do not require matrix inversions.  
A second approach is to use a fix-point iteration as in \cite{MR3392626,MR3148086}. Unfortunately, 
this iteration is not guaranteed to converge. 
A third approach (see \cite{GVW-dual,Gomes:2014kq}) is to
solve the master equation, which is a partial differential equation  whose characteristics are given by \eqref{MFM:cont_dyn}. 
To approximate the master equation, 
we can use a finite-difference method constructed by solving $N$-player problem. Unfortunately, even for a modest number of states, this approach 
is computationally expensive. 

Our approach to the numerical solution of \eqref{MFM:cont_dyn} relies on the monotonicity of the operator, $A:\Rr^d\times\Rr^d\to \Rr^d\times \Rr^d$, 
given by
\begin{equation}
\label{opa}
A\left[
\begin{array}{c}
\theta\\
u
\end{array}\right]
=
\left[
\begin{array}{c}
h(\Delta_i u,\theta,i)\\
-\sum_{j} \theta^j \alpha^*_i(\Delta_j u,\theta,j) 
\end{array}
\right].                        
\end{equation} 
More precisely, we assume 
that $A$ is monotone (see Assumption \ref{A2}) in the sense that
\[
\left(
A\left[
\begin{array}{c}
\theta\\
u
\end{array}\right]-
A\left[
\begin{array}{c}
\tilde \theta\\
\tilde u
\end{array}\right],
\left[
\begin{array}{c}
\theta\\
u
\end{array}\right]
-
\left[
\begin{array}{c}
\tilde \theta\\
\tilde u
\end{array}\right]
\right)\geq 0
\]
for all $\theta, \tilde \theta\in \Ss^d$ and $u, \tilde u\in \Rr^d$. 
Building upon the ideas in \cite{AFG} for stationary problems (also see the approaches for stationary problems in
\cite{MR3420414,FS, AMFS, MR3597009}), 
we introduce the flow
\begin{equation}
\label{pmf}
\left[
\begin{array}{c}
\theta_s\\
u_s
\end{array}\right]
=-A\left[
\begin{array}{c}
\theta\\
u
\end{array}\right].  
\end{equation}
Up to the normalization of $\theta$, the foregoing flow is a contraction provided $\theta\in \Ss^d$. Moreover, its fixed points solve
\[
A\left[
\begin{array}{c}
\theta\\
u
\end{array}\right]=0.
\]
In Section \ref{sec:stationary_problem}, we construct a discrete version of \eqref{pmf} that preserves probabilities; 
that is, both the total mass of $\theta$ and its non-negativity. 

The time-dependent case is substantially more delicate and, hence, our method
to approximate its solutions is  
the main contribution of this paper.
The operator associated with the time-dependent problem, 
$A:H^1(0,T; \Rr^d\times \Rr^d)\to  L^2(0,T; \Ss^d\times \Rr^d)$, is
\begin{equation}
\label{opb}
A\left[
\begin{array}{c}
\theta\\
u
\end{array}\right]
=
\left[
\begin{array}{c}
-u_t+h(\Delta_i u,\theta,i)\\
\theta_t-\sum_{j} \theta^j \alpha^*_i(\Delta_j u,\theta,j) 
\end{array}
\right].                        
\end{equation} 
Under the initial-terminal condition in \eqref{itc}, $A$ is a monotone operator. Thus, the  flow 
\begin{equation}
\label{mf}
\left[
\begin{array}{c}
\theta_s\\
u_s
\end{array}\right]
=-A\left[
\begin{array}{c}
\theta\\
u
\end{array}\right]  
\end{equation}
for  $(\theta, u)\in L^2(0,T; \Rr^d\times \Rr^d)$
is formally a contraction. 
 Unfortunately, even if this flow is well defined, 
the preceding system does not preserve probabilities nor the boundary conditions \eqref{itc}. Thus, 
in Section \ref{sec:mean_field_solution}, we modify \eqref{mf} in a way that it becomes a contraction in $H^1$ and preserves the boundary
conditions. 
Finally, we discretize this modified flow and build a numerical
algorithm to approximate solutions of \eqref{MFM:cont_dyn}-\eqref{itc}. 
Unlike Newton-based methods, our algorithm does not need the inversion
of large matrices and scales linearly with the number of states. This is
particularly relevant for finite-state MFGs that arise from the 
discretization of continuous-state MFGs. 
We illustrate our results in a paradigm-shift problem introduced in \cite{BesancenotDogguy} and
studied from a numerical perspective in \cite{Gomes:2014kq}.

We conclude this introduction with a brief outline of the paper. 
In the following section, we discuss the framework, main assumptions, 
and the paradigm-shift example that illustrates our methods. 
Next, we address stationary solutions. Subsequently, in section 
\ref{sec:mean_field_solution}, we discuss the main contribution of this paper by addressing the initial-terminal value problem. There, 
we outline the projection method, explain its discretization, and 
present numerical results. The paper ends with a brief concluding section. 

\section{Framework and main assumptions} 
\label{assp}
Following \cite{GMS2}, we present the standard finite-state MFG framework and describe
our main assumptions.
Then, we discuss 
a paradigm-shift problem from \cite{BesancenotDogguy} that we
use to illustrate our methods.   

\subsection{Standard setting for finite-state MFG} 
\label{sub:standard_setting_for_finite_state_mfg}

Finite-state MFGs model systems with many identical players who act rationally and non-cooperatively. 
These players switch between states in $I_d$ seeking to minimize a cost.
Here, the macroscopic state of the game is a probability vector $\theta\in \Ss^d$ that gives the players' distribution in $I_d$. 
A typical player controls
the switching rate, $\alpha_j(i)$, from its state, $i\in I_d$, to a new state, $j\in I_d$.
Given the players' distribution $\theta(r)$ at time $r$, each player chooses a non-anticipating control, $\alpha$, that minimizes the cost
\begin{equation}
\label{ui}
u^i(t;\alpha)=E_{\bi_t=i}^{\alpha} \left[\int_t^T c(\bi_r,\theta(r),\alpha(r)) dr + u^{\bi_T}(\theta(T))\right]. 
\end{equation}
In the preceding expression,  $c:I_d\times \Ss^d\times(\Rr_0^+)^d\to \Rr$ is a running cost, $\Psi\in \Rr^d$ the terminal cost, and $\bi_s$ is a Markov process in $I_d$ with switching rate $\alpha$. 
The \emph{Hamiltonian}, $h$,  is the generalized Legendre transform of  $c(i,\theta,\cdot)$: 
        \begin{equation*}
                h(\Delta_i z,\theta,i) = \min_{\mu \in(\Rr^+_0)^d} \{ c(i,\theta,\mu) + \mu \cdot \Delta_i z\}.
        \end{equation*}
        
The first equation in \eqref{MFM:cont_dyn} determines the value function $u$ for \eqref{ui}. 
The optimal switching rate from state $i$ to state $j\neq i$ is given by
$\alpha^*_j(\Delta_iu,\theta,i)$, 
where
\begin{equation}
\label{alfstar}
        \alpha^*_j(z,\theta,i)= \argmin_{\mu \in(\Rr^+_0)^d} \{c(i,\theta,\mu)+\mu \cdot \Delta_i z\}.
\end{equation}
Moreover, at points of differentiability of $h$, we have \eqref{alfstar2}. 
The rationality of the players implies that each of them chooses the optimal switching rate, $\alpha^*$. 
Hence, $\theta$ evolves according to the second equation in \eqref{MFM:cont_dyn}.
\subsection{Main assumptions} 
\label{sub:the_monotonicity_condition}

Because we work with the Hamiltonian, $h$, rather than the running cost, $c$, it is convenient to state our assumptions in terms of the former. 
For the relation between assumptions on $h$ and $c$, see \cite{GMS2}.

We begin by stating a mild assumption that ensures the existence of solutions 
for \eqref{MFM:cont_dyn}.
\begin{hyp}
\label{A1}
The Hamiltonian $h(z,\theta, i)$ is locally Lipschitz in $(z, \theta)$,
differentiable in $z$, 
and the map $z\mapsto h(z, \theta, i)$ is concave for each $(\theta, i)$. 

The function $\alpha^*(z, \theta, i)$ given by \eqref{alfstar2} is locally Lipschitz.
\end{hyp}       

Under the previous Assumption, there exists a solution to \eqref{MFM:cont_dyn}-\eqref{itc}, see \cite{GMS2}. This solution may not be unique as the examples in \cite{GVW-dual} and \cite{Gomes:2014kq} show.
Monotonicity conditions are commonly used in MFGs  to prove
uniqueness of solutions. 
For finite-state MFGs, the appropriate monotonicity condition is stated in the next Assumption. Before proceeding, we define $\|v \|_{\sharp}=\inf_{\lambda\in \re}\|v +\lambda\mathbf{1} \|$.
\begin{hyp}
        \label{A2}
        There exists $\gamma>0$\ such that the Hamiltonian, $h$, satisfies the following monotonicity property
        \begin{align*}
                &\theta \cdot (h(z,\tilde \theta)-h(z,\theta))+ \tilde \theta\cdot(h(\tilde z,\theta)-h(\tilde z,\tilde \theta)   )\leq -\gamma \|\theta -\tilde \theta  \|^2.
        \end{align*}
        Moreover, for each $M>0$, there exist constants $\gamma_i$ such that on the set  
        $\|w\|,\|z\|_\sharp\leq M$, $h$ satisfies the following concavity property
        \begin{equation*}
                h(z,\theta,i)-h( w,\theta,i)-\alpha^*(w,\theta,i)\cdot\Delta_i(z-w)\leq -\gamma_i\|\Delta_i (z-w)\|^2.
        \end{equation*}         
\end{hyp}               
Under the preceding assumptions,             
\eqref{MFM:cont_dyn}-\eqref{itc} has a unique solution, see \cite{GMS2}.
Here, the previous condition is essential to the convergence of our numerical methods, both for stationary problems, in Section \ref{sec:stationary_problem}, 
and for the general time-dependent case, in Section \ref{sec:mean_field_solution}. 
\begin{rem}
\label{R1}      
As shown in \cite{GMS2}, 
        Assumption \eqref{A2} implies the
        inequality
        \begin{align*}
                &\sum_{i=1}^d (u^i-\tilde u^i) \lb \sum_j \theta^j \alpha^*(\Delta_j u,\theta,j) - \sum_j \tilde \theta^j \alpha^*(\Delta_j \tilde u, \tilde \theta,j) \rb \\
                &+ \sum_{i=1}^d  (\theta^i-\tilde \theta^i) \lb - h(\Delta_i u, \theta,i) + k +h(\Delta_i \tilde u, \tilde \theta,i) -\tilde k\rb\\
                &\leq -\gamma \|(\theta- \tilde \theta)(s)\|^2 - \sum_{i=1}^d \gamma_i(\theta^i+\tilde \theta^i)(s) \|(\Delta_i u- \Delta_i \tilde u)(s)\| ^2 
        \end{align*}
        for any $u, \tilde u\in \Rr^d$,  $\theta, \tilde \theta\in \Ss^d$, and $k, \tilde k\in \Rr$. 
\end{rem}

\subsection{Solutions and weak solutions}

Because the operator $A$ in \eqref{opb} is monotone, we have a natural
concept of weak solution for \eqref{MFM:cont_dyn}-\eqref{itc}. 
These weak solutions were considered for continuous-state MFGs in \cite{AFG} and in \cite{FG2}.
We say that $(u, \theta)\in L^2((0,T), \Rr^d)\times L^2((0,T), \Ss^d) $ is a weak solution of \eqref{MFM:cont_dyn}-\eqref{itc} if
for all  $(\tilde u, \tilde \theta)\in H^1((0,T), \Rr^d)\times H^1((0,T), \Ss^d)$ satisfying \eqref{itc}, 
we have
\[
\left\langle
A\left[
\begin{array}{c}
\tilde \theta\\
\tilde u
\end{array}\right], 
\left[
\begin{array}{c}
\tilde \theta-\theta\\
\tilde u-u
\end{array}\right]
\right\rangle\geq 0.
\]
Any solution of \eqref{MFM:cont_dyn}-\eqref{itc} is a weak solution, and any sufficiently regular weak solution with $\theta>0$ is a solution. 

Now, we turn our attention to the stationary problem.
        We recall, see \cite{GMS2}, that a stationary solution of \eqref{MFM:cont_dyn} is a triplet $(\bar \theta,\bar u, \bar k) \in \Ss^d \times \Rr^d \times \Rr$ satisfying
        \begin{equation} \label{stat_mfg}
        \begin{cases}
        h(\Delta_i \bar u, \bar \theta, i) = \bar k \\
        \sum_{j} \bar \theta^j \alpha^*_i (\Delta_j \bar u, \bar \theta,j) = 0
        \end{cases}
        \end{equation}
        for $i=1,\ldots,d$.
As discussed in \cite{GMS2},
the existence of solutions to \eqref{stat_mfg} holds under an additional 
contractivity assumption. In general, as for continuous-state MFGs, solutions for \eqref{stat_mfg} 
may not exist. Thus, we need to consider weak solutions.
For a finite-state MFG, a weak solution
of \eqref{stat_mfg} is a triplet
$(\bar u, \bar \theta, \bar k)\in \Rr^d\times\Ss^d\times \Rr$ that satisfies
        \begin{equation} \label{stat_mfg_ws}
        \begin{cases}
        h(\Delta_i \bar u, \bar \theta, i) \geq \bar   k \\
        \sum_{j} \bar \theta^j \alpha^*_i (\Delta_j \bar u, \bar \theta,j)
                = 0
        \end{cases}
        \end{equation}
        for $i=1,\ldots,d$, with equality in the first equation for all indices ~$i$ such that $\bar \theta^i>0$.

        \subsection{Potential MFGs} 
        \label{sub:potential_mean_field_games}
        In a potential MFG, the Hamiltonian is of the form
        \begin{equation*}
        h(\nabla_i u, \theta, i) = \tilde h(\nabla_i u,i) + f(\theta,i),
        \end{equation*}
        with $\tilde h: \Rr^d \times I_d \to \Rr$, $f:\Rr^d \times I_d \to \Rr$ and $f$ is the gradient of a convex function, 
        $F:\Rr^d \to \Rr$; that is,  $f(\theta,\cdot) = \nabla_{\theta}F(\theta)$.
        We define  $H:\Rr^d \times \Rr^d \to \Rr$ as
        \begin{equation}
        \label{hfor}
        H( u, \theta) = \sum_{i=1}^d \theta^i \tilde h(\nabla_{i} u,i) + F(\theta).
        \end{equation}
Then, \eqref{MFM:cont_dyn} can be written in Hamiltonian form as
\[
\begin{cases}
u_t=-D_\theta H(u, \theta)\\
\theta_t=D_u H(u, \theta). 
\end{cases}
\]
In particular,  $H$ is conserved:
\[
\frac{d}{dt}H(u,\theta)=0. 
\]
In section \ref{sec:numerical_implementation}, 
we use this last property as an additional test for our numerical method.

\subsection{A case study -- the paradigm-shift problem} 
\label{sub:a_case_study_the_paradigm_shift_problem}

        A paradigm shift is a change in a fundamental assumption within the ruling theory of science. 
                Scientists or researchers can work in multiple competing theories or problems. 
        Their choice seeks to maximize the  recognition (citations, awards, or prizes) and scientific activity (conferences or collaborations, for example). 
 This problem was formulated as a two-state MFG in \cite{BesancenotDogguy}. 
                 Subsequently, it was studied numerically in \cite{Gomes:2014kq} and \cite{GVW-dual} using a $N$-player approximation and PDE methods. 
                 Here, we present the stationary and time-dependent versions of this problem. 
       Later,  we use them to illustrate our numerical methods.  
                
        We consider the running cost $c:I_d\times \Ss^d \times (\Rr_0^+)^2\to \Rr$ given by
        \begin{equation*}
                c(i,\theta,\mu) = f(i,\theta) + c_0(i,\mu), \text{ where } c_0(i,\mu) = \frac{1}{2} \sum_{j\neq i}^2 \mu^2_j.
        \end{equation*}
        The functions $f=f(i,\theta)$ are \emph{productivity functions with constant elasticity of substitution}, given by
        \begin{align*}
        \begin{cases}
                f(1,\theta) = \lb a_1 (\theta^1)^r + (1-a_1)(\theta^2)^r\rb^{\frac{1}{r}}\\
                f(2,\theta) = \lb a_2 (\theta^1)^r+ (1-a_2)(\theta^2)^r\rb^{\frac{1}{r}} 
        \end{cases}        
        \end{align*}
        for $r\geq 0$ and $0\leq a_1, a_2\leq 1$.
        The Hamiltonian is 
        \begin{align*}
        \begin{cases}
                h(u,\theta,1) = f(1,\theta) - \frac{1}{2}\lb (u^1-u^2)^+\rb^2,\\
                h(u,\theta,2) = f(2,\theta) - \frac{1}{2}\lb (u^2-u^1)^+\rb^2,
        \end{cases}        
        \end{align*}
        and the optimal switching rates are 
        \begin{align*}
                \alpha_2^*(u,\theta, 1) =  (u^1-u^2)^+, \quad \alpha_1^*(u,\theta, 1) = -(u^1-u^2)^+,\\
                \alpha_1^*(u,\theta, 2) =  (u^2-u^1)^+, \quad \alpha_2^*(u,\theta, 2) = -(u^2-u^1)^+.
        \end{align*} 
                
        For illustration, we examine the case where
        $a_1=1$, $a_2=0$, and $r=1$ in the productivity functions above. In this case, $f=\nabla_\theta F(\theta)$ with \
\[
F(\theta)=\frac{(\theta^1)^2+(\theta^2)^2}{2}. 
\]
Moreover, the game is potential with
\[
H(u, \theta)=- \frac{1}{2}\lb (u^1-u^2)^+\rb^2\theta^1- \frac{1}{2}\lb (u^2-u^1)^+\rb^2\theta^2 +F(\theta).
\]
Furthermore, $(\bar \theta, \bar u,k)$ is a stationary solution if it solves
        \begin{equation} \label{sys:sta_sol_hj}
                \begin{cases}
                {\theta^1}-\frac{1}{2}((u^1-u^2)^+)^2=k\\
                {\theta^2}-\frac{1}{2}((u^2-u^1)^+)^2=k,
                \end{cases}
        \end{equation}
        and
        \begin{equation} \label{sys:stat_sol_kolm}
                \begin{cases}
                -\theta^1(u^1-u^2)^+ + \theta^2 (u^2-u^1)^+=0\\
                \theta^1(u^1-u^2)^+ - \theta^2 (u^2-u^1)^+=0.
                \end{cases}
        \end{equation}
        Since $\theta^1+\theta^2=1$, and using the symmetry of \eqref{sys:sta_sol_hj}-\eqref{sys:stat_sol_kolm}, we conclude that 
        \begin{equation}
        \label{ssf}
                (\bar \theta, \bar u, k) = \lb \lb\frac{1}{2},\frac{1}{2}\rb,(p,p), \frac{1}{2} \rb, \quad p \in \Rr.
        \end{equation}
                
        The time-dependent paradigm-shift problem 
        is determined by 
        \begin{equation} \label{sys:td_sol_hj}
                \begin{cases}
                        u^1_t=-{\theta^1}+\frac{1}{2}((u^1-u^2)^+)^2\\
                        u^2_t=-{\theta^2}+\frac{1}{2}((u^2-u^1)^+)^2,
                \end{cases}
        \end{equation}
        and
        \begin{equation} \label{sys:td_sol_kolm}
                \begin{cases}
                        \theta^1_t=-\theta^1(u^1-u^2)^+ + \theta^2 (u^2-u^1)^+\\
                        \theta^2_t=\theta^1(u^1-u^2)^+ - \theta^2 (u^2-u^1)^+, 
                \end{cases}
        \end{equation}
        together with initial-terminal conditions 
        \begin{equation*}
                \theta^i(0) = \theta_0, \text{ and } u^i(T) = u_T^i
        \end{equation*}
        for $i=1,2$, $\theta_0\in \Ss^2$, and $u_T\in \Rr^2$. 

\section{Stationary problems} 
\label{sec:stationary_problem}
    To approximate the solutions of \eqref{stat_mfg}, we introduce a flow closely related to \eqref{pmf}. 
    This flow is the analog for finite-state problems of the one considered in \cite{AFG}. 
    The monotonicity in Assumption \ref{A2} gives the contraction property. 
    Then, we construct a numerical algorithm using an Euler step combined with a projection step to ensure that $\theta$ remains a probability. 
    Finally, we test our algorithm in the paradigm-shift model. 
        \subsection{Monotone approximation} 
        \label{sub:contractive_auxiliary_dynamics}
        To preserve the mass of $\theta$, we introduce the following modification of \eqref{pmf} 
        \begin{equation} \label{eq:aux_dyn_stat}
                \begin{cases}
                u_s^i      = \sum_{j} \theta^j \alpha^*_i(\Delta_j u,\theta,j) \\
                \theta_s^i = -h(\Delta_i u,\theta,i)+k(s),
            \end{cases}
        \end{equation}
        where $k:\Rr_0^+\to \Rr$ is such that  $\sum_{i=1}^d \theta^i(s) = 1$ for every $s\geq 0$. 
                For this condition to hold, we need $\sum_{i=1}^d \theta_s^i=0$. 
        Therefore,                 
                \begin{equation} \label{eq:k_aux}
                k(s) = \frac{1}{d} \sum_{i=1}^d h(\Delta_i u,\theta,i).
        \end{equation}
        \begin{pro} \label{prop:contraction_sta_sol}
                Suppose that Assumptions \ref{A1}-\ref{A2} hold. 
            Let $(u,\theta)$ and $(\tilde u, \tilde\theta)$ solve \eqref{eq:aux_dyn_stat}-\eqref{eq:k_aux}. 
                        Assume that $\sum_i \theta^i(0)= \sum_i \tilde \theta^i(0)=1$ and  that $\theta(s), \tilde \theta(s)\geq 0$. 
                        Then, 
            \begin{align*}
                &\frac{d}{ds} \lb \|(u-\tilde u) \|^2 + \|\theta-\tilde \theta\|^2  \rb \\&\quad\leq -\gamma \|(\theta- \tilde \theta)(s)\|^2 
                                                - \sum_{i=1}^d \gamma_i(\theta^i+\tilde \theta^i)(s) \|(\Delta_i u- \Delta_i \tilde u)(s)\| ^2.
            \end{align*}
                \end{pro}
        \begin{proof}
                We begin with the identity
                \begin{align*}
                &\frac{1}{2} \frac{d}{ds} \sum_{i=1}^d \lsb (u^i-\tilde u^i)^2  + (\theta^i- \tilde \theta^i)^2 \rsb \\&
                \qquad= \sum_{i=1}^d (u^i-\tilde u^i) (u^i-\tilde u^i)_s 
                      + (\theta^i-\tilde \theta^i) (\theta^i-\tilde \theta^i)_s.
            \end{align*}
            Using \eqref{eq:aux_dyn_stat} in the previous equality, we obtain
            \begin{align*}
                \frac{1}{2} \frac{d}{ds} &\sum_{i=1}^d \lsb (u^i-\tilde u^i)^2  + (\theta^i- \tilde \theta^i)^2 \rsb\\ 
                        &= \sum_{i=1}^d (u^i-\tilde u^i) \lb \sum_j \theta^j \alpha^*(\Delta_j u,\theta,j) - \sum_j \tilde \theta^j \alpha^*(\Delta_j \tilde u, \tilde \theta,j) \rb \\
                        &+ \sum_{i=1}^d  (\theta^i-\tilde \theta^i) \lb - h(\Delta_i u, \theta,i) + k +h(\Delta_i \tilde u, \tilde \theta,i) -\tilde k\rb\\
                        &\leq -\gamma \|(\theta- \tilde \theta)(s)\|^2 - \sum_{i=1}^d \gamma_i(\theta^i+\tilde \theta^i)(s) \|(\Delta_i u- \Delta_i \tilde u)(s)\| ^2,
            \end{align*}
            by Remark \ref{R1}.             
                \end{proof}

\subsection{Numerical algorithm}

Let $A$ be given by \eqref{opa}. Due to the monotonicity, for $\mu$ small, the Euler map,
\[
E_\mu\left[
\begin{array}{c}
\theta\\
u
\end{array}
\right]
=
\left[
\begin{array}{c}
\theta\\
u
\end{array}
\right]
-\mu A\left[
\begin{array}{c}
\theta\\
u
\end{array}
\right], 
\]
is a contraction, provided that $\theta$ is a non-negative probability vector. However, $E_\mu$ may not keep $\theta$ non-negative
and, in general, $E_\mu$ also does not preserve the mass. Thus, we introduce the following  projection operator on $\Ss^d\times\Rr^d$:
\[
P
\left[
\begin{array}{c}
\theta\\
u
\end{array}
\right]
=
\left[
\begin{array}{c}
\varpi(\theta)\\
u
\end{array}
\right], 
\]
where $\varpi(\theta)_i=(\theta^i+\xi)^+$ and $\xi$ is such that 
\[
\sum_i \varpi(\theta)_i=1. 
\]
Clearly, $P$ is a contraction because it is a projection on a convex set. Finally, to approximate stationary solutions of \eqref{stat_mfg}, we consider the iterative map
\begin{equation}
\label{iterp}
\left[
\begin{array}{c}
\theta_{n+1}\\
u_{n+1}
\end{array}
\right]
=
P E_\mu
\left[
\begin{array}{c}
\theta_{n}\\
u_{n}
\end{array}
\right]. 
\end{equation}
We have the following result:

\begin{pro}
Let $(\bar \theta, \bar u, \bar k)$ solve \eqref{stat_mfg_ws}. Then, $(\bar \theta, \bar u )$ is a fixed point for \eqref{iterp}. Moreover, for any fixed point of  \eqref{iterp}, there exists $\bar k$ such that 
$(\bar \theta, \bar u, \bar k)$ solves  \eqref{stat_mfg_ws}.
 
Finally, if $\mu$ is small enough and \eqref{stat_mfg_ws} has a weak solution, $(\bar \theta, \bar u, \bar k)$, then
the iterates in \eqref{iterp} are bounded and converge to  $(\bar \theta, \bar u )$.
\end{pro}
\begin{proof}
Clearly, a weak solution of  \eqref{stat_mfg} is a fixed point for \eqref{iterp}. Conversely, let  $(\bar \theta, \bar u )$ be a fixed point for \eqref{iterp}. Then, 
\[
\bar u^i=\bar u^i+\mu \sum_{j} \bar \theta^j \alpha^*_i (\Delta_j \bar u, \bar \theta,j).
\]
Hence, \[\sum_{j} \bar \theta^j \alpha^*_i (\Delta_j \bar u, \bar \theta,j)=0.\]
Additionally, we have
\[
\bar \theta^i=\left(\bar \theta^i- \mu  h(\Delta_i \bar u, \bar \theta, i) +\xi\right)^+ \]
for some $\xi$. Thus, for $\bar k=\frac \xi \mu$, 
\[
        h(\Delta_i \bar u, \bar \theta, i) \geq   \bar k, \\
\]
with equality when $\bar \theta^i>0$.

If $\mu$ is small enough, $E_\mu$ is a contraction because $A$\ is a monotone Lipschitz map.  Thus,  if there is a solution of  \eqref{stat_mfg_ws}, the iterates in  \eqref{iterp} are bounded. Then, the convergence follows from the strict monotonicity of  $E_\mu$.      
\end{proof}

\subsection{Numerical examples}

To illustrate our algorithm, we consider the paradigm-shift problem. The monotone flow in \eqref{eq:aux_dyn_stat} is
                \begin{equation} \label{ums}
                        \begin{cases}
                                u^1_s = -\theta^1(u^1-u^2)^+ +\theta^2 (u^2-u^1)^+\\
                                u^2_s = \theta^1(u^1-u^2)^+ -\theta^2 (u^2-u^1)^+,
                        \end{cases}
                \end{equation}
                and
                \begin{equation} \label{tms}
                        \begin{cases}
                                \theta^1_s = -\theta^1 + \frac{1}{2}((u^1-u^2)^+)^2+k(s)\\
                                \theta^2_s = -\theta^2 + \frac{1}{2}((u^2-u^1)^+)^2+k(s). 
                        \end{cases}
                \end{equation}
                According to   \eqref{eq:k_aux},  
                \begin{equation*}
                        k(s) = \frac{1}{2} \lb \theta^1 - \frac{1}{2}((u^1-u^2)^+)^2
+  \theta^2 - \frac{1}{2}((u^2-u^1)^+)^2 \rb.
                \end{equation*}

 Now, we present the numerical results for this model using the iterative method in  \eqref{iterp}. 
{ We set $s\in[0,8]$ and discretize this interval into $N=300$ subintervals.} 
First, we consider the following initial conditions:
        \begin{equation*}
        u^1_0=4, \,  u^2_0=2  \text{ and } \theta^1_0= 0.8, \, \theta^2_0=0.2.
        \end{equation*}
        The convergence towards the stationary solution is illustrated in Figures \ref{fig:stat:theta} and   \ref{fig:stat:valuef} for $\theta$ and $u$.   The behavior of $k$ is shown in Figure \ref{fig:stat:const}.
        In Figure
\ref{fig:L2_norm}, we illustrate the contraction of the norm
\[
\left\|
\left[
\begin{array}{c}
\theta(s)\\
u(s)
\end{array}
\right]
-
\left[
\begin{array}{c}
\bar \theta\\
\bar u
\end{array}
\right]
\right\|, 
\]
where $(\bar \theta, \bar u)$ is the stationary solution in \eqref{ssf}.
        \begin{figure}
        	\centering
        	\begin{subfigure}{0.49\textwidth}
        		\centering
        		\includegraphics[width=1\linewidth]{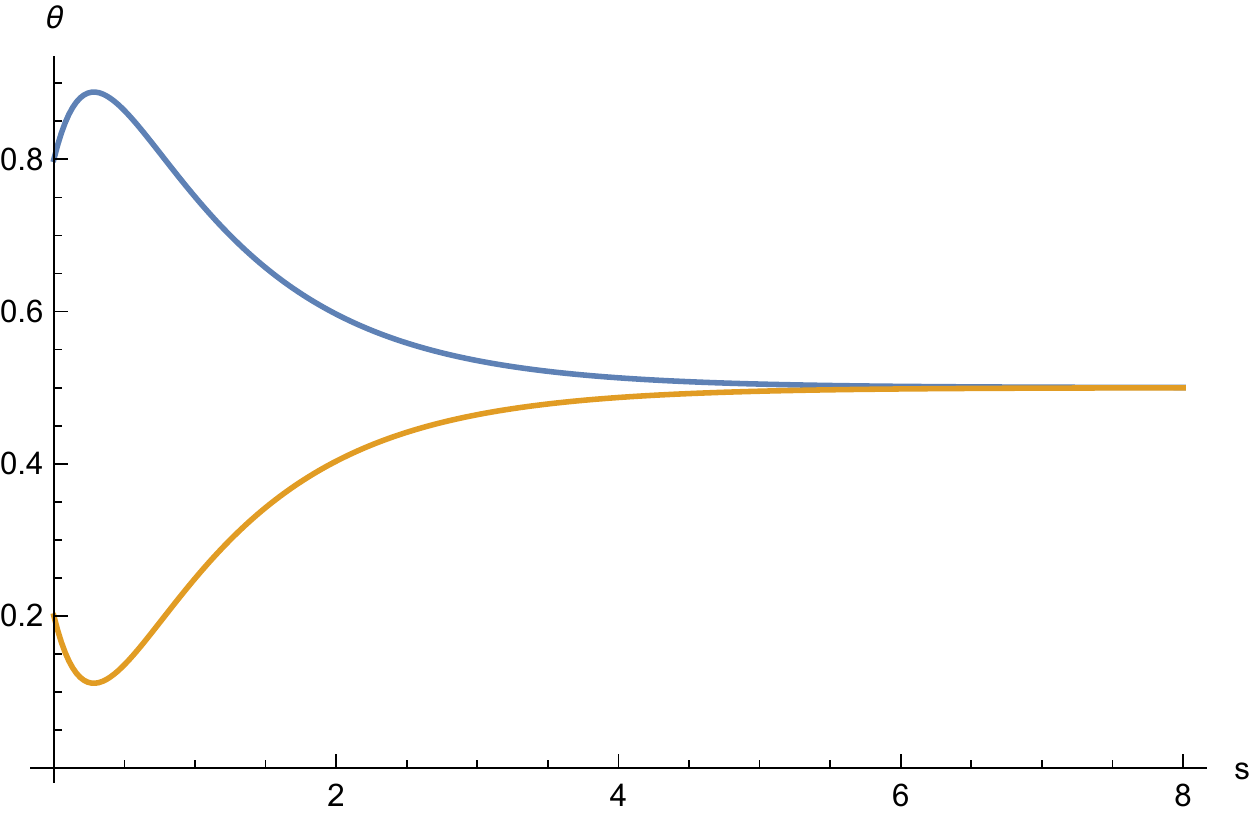}
        		\caption{Convergence of  $\theta^i$.}
        		\label{fig:stat:theta}
        	\end{subfigure}
        	\begin{subfigure}{0.49\textwidth}
        		\centering
        		\includegraphics[width=1\linewidth]{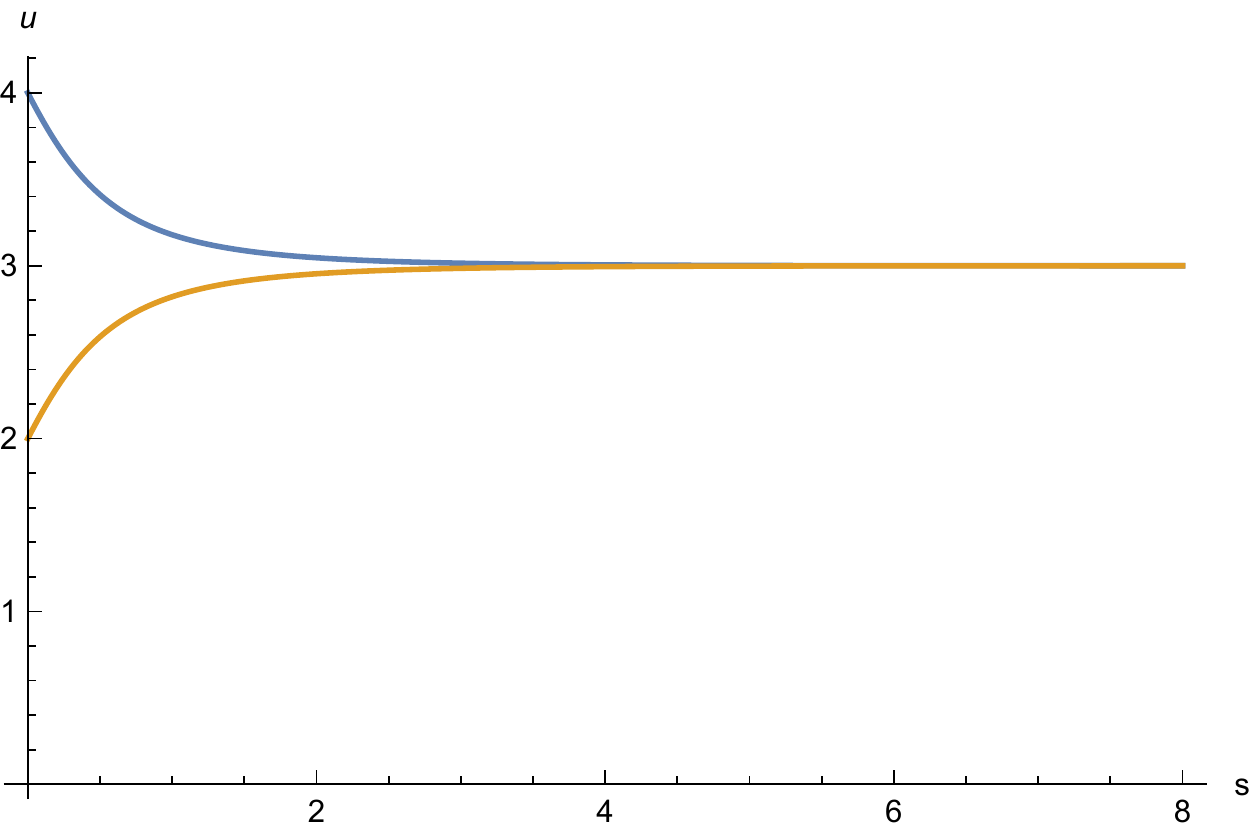}
        		\caption{Convergence of  $u^i$.}
        		\label{fig:stat:valuef}
        	\end{subfigure}
        	\caption{Evolution of $\theta$ and $u$ with the monotone flow, for $s \in [0,8]$.}
        \end{figure}      
                \begin{figure}
                \centering
                \begin{subfigure}{0.5\textwidth}
                        \centering
                        \includegraphics[width=1\linewidth]{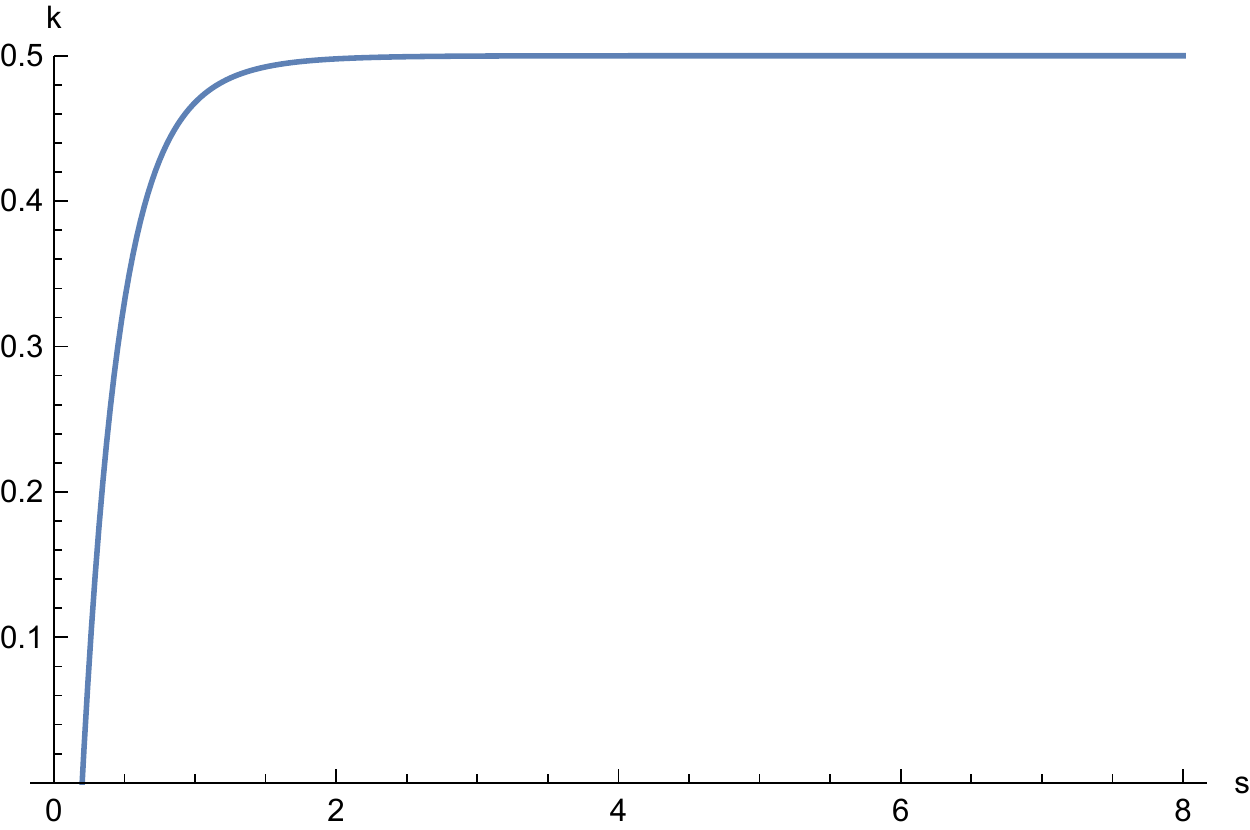}
                        \caption{Evolution of $k$.}
                        \label{fig:stat:const}
                \end{subfigure}
                \begin{subfigure}{0.5\textwidth}
                        \centering
                        \includegraphics[width=1\linewidth]{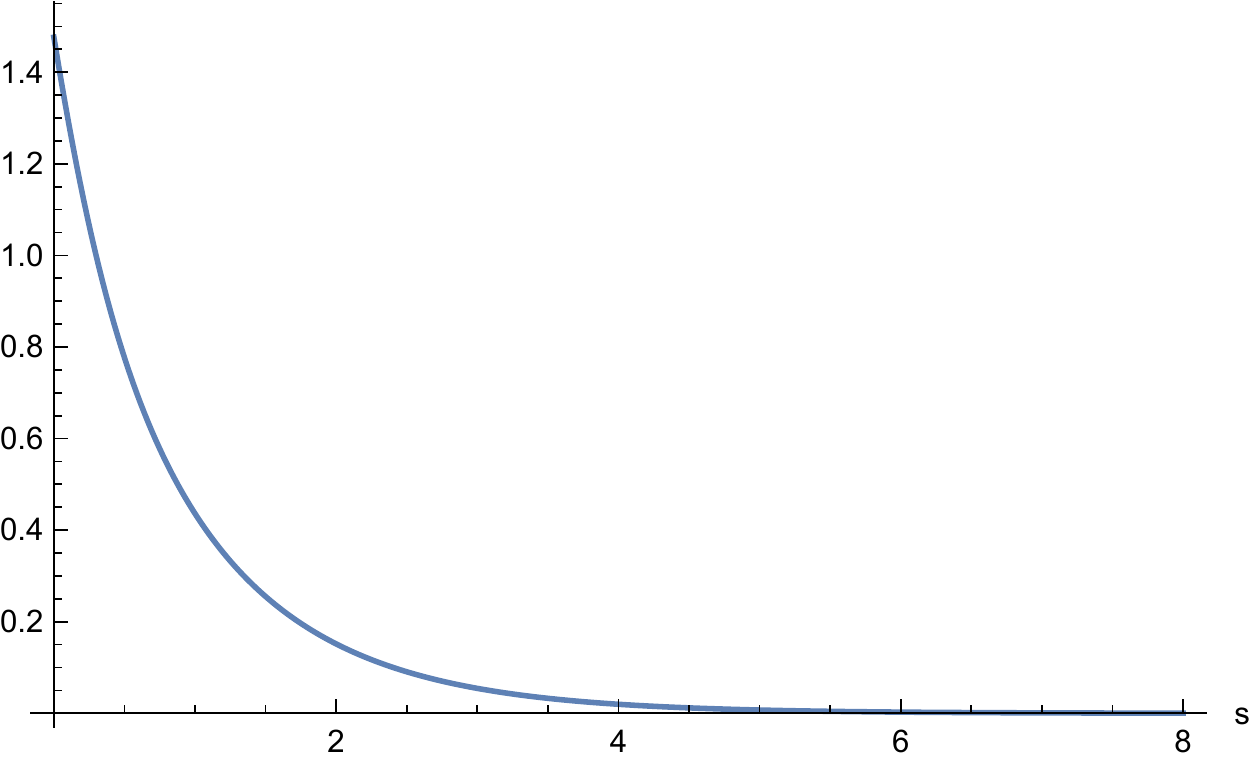}
                        \caption{Contraction of the norm.}
                        \label{fig:L2_norm}
                \end{subfigure}
                \caption{Evolution of $k$ and norm contraction, $\|(\theta,u)-(\bar\theta,\bar u)\|$.}
                \label{fig:k_L2_norm}
        \end{figure}
        Next, we consider the case where the iterates of $E_\mu$ do not preserve positivity.  
        In Figure \ref{fig:Positivity_comp}, we compare the evolution of $\theta$ by iterating  $E_\mu$, without the projection, and using \eqref{iterp}.
        In the first case,  $\theta$ may not remain positive, although, in this example, convergence holds. In Figure \ref{fig:Positivity_comp},
        we plot the evolution through  \eqref{iterp} of $\theta$ towards the analytical solution $\theta^1=\theta^2=0.5$.
        \begin{figure}
                \centering
                \begin{subfigure}{0.5\textwidth}
                        \centering
                        \includegraphics[width=1\linewidth]{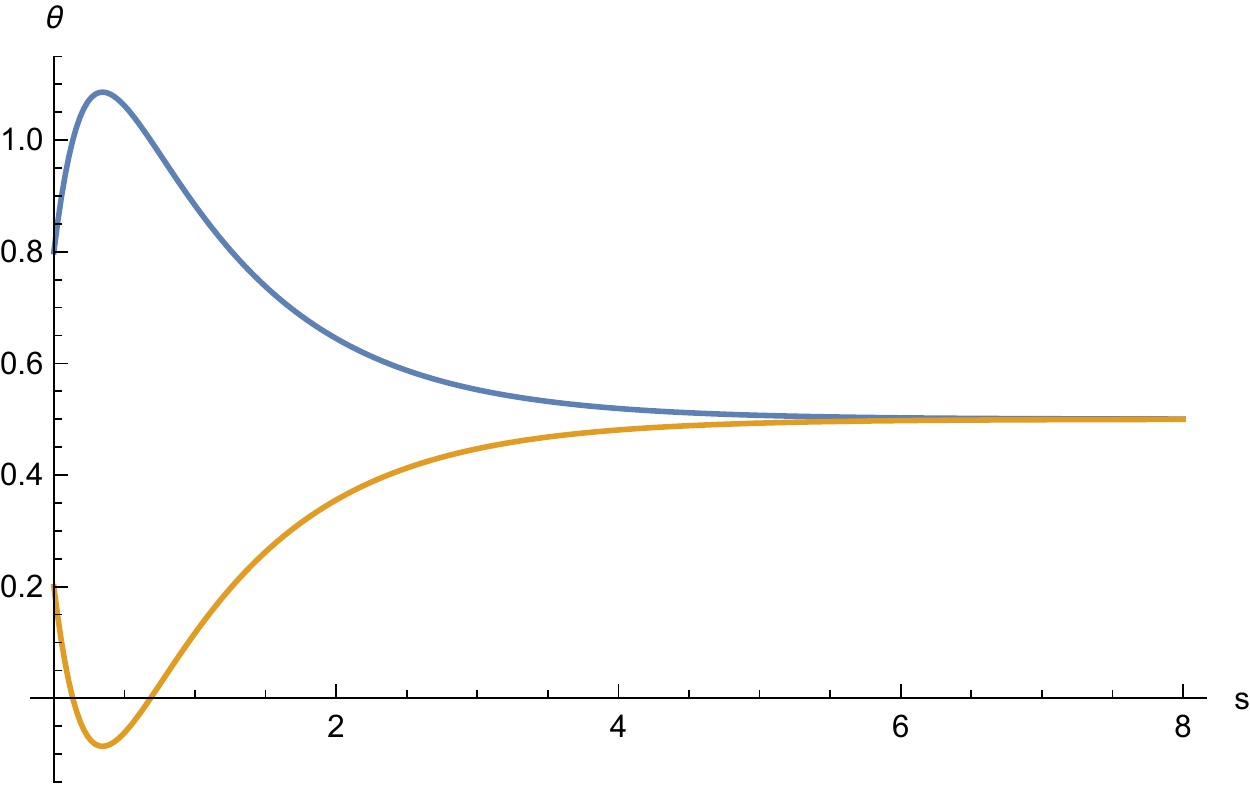}
                        \caption{Non positivity of the distribution using $E_\mu$.}
                \end{subfigure}
                \begin{subfigure}{0.5\textwidth}
                        \centering
                        \includegraphics[width=1\linewidth]{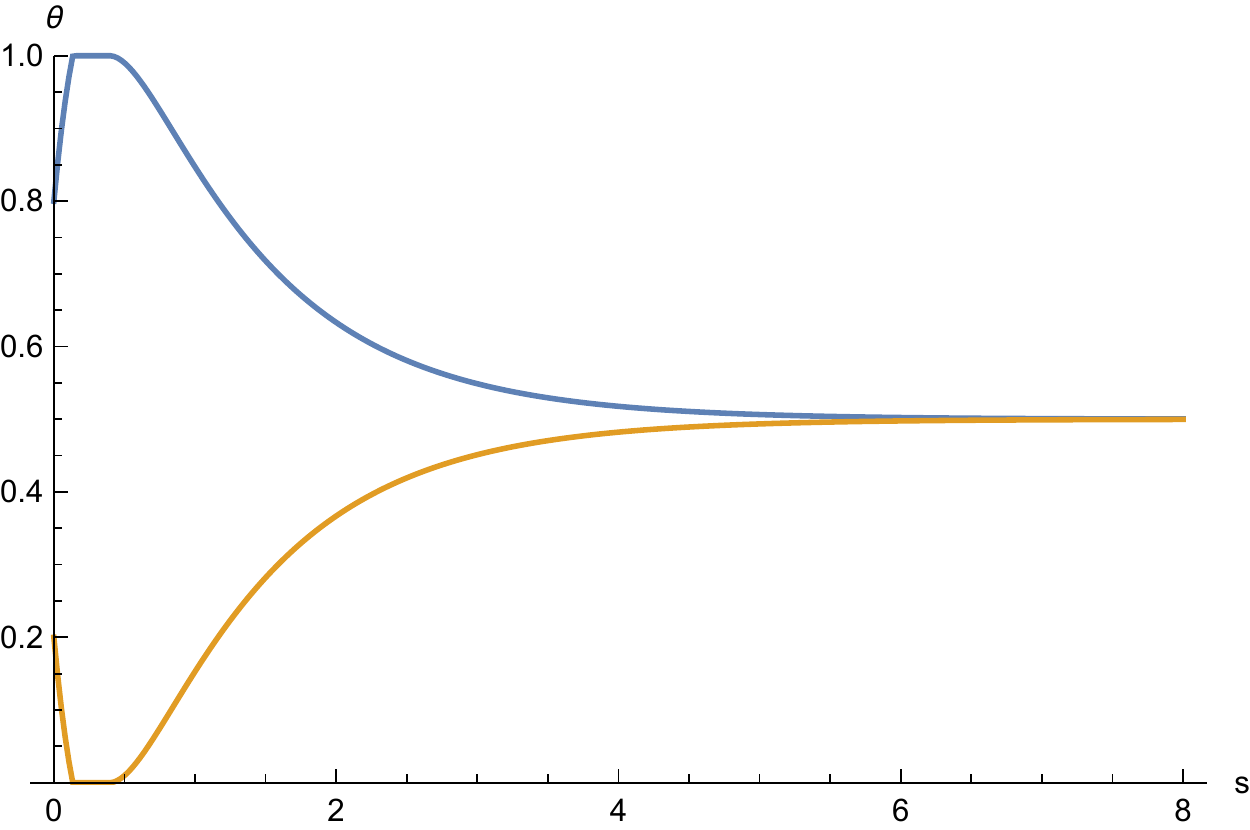}
                        \caption{Convergence using \eqref{iterp} while preserving the positivity of $\theta$.}
                \end{subfigure}
                \caption{Comparison between the iterates of $E_\mu$ and $PE_\mu$ for $\theta^1_0=0.8$, $\theta^2_0=0.2$, 
                        $u^1_0=5$, and $u^2_0=2$.}
                \label{fig:Positivity_comp}
        \end{figure}
        As expected from its construction, $\theta$ is always non-negative and a  probability. The contraction of the norm is similar to the 
        previous case, see Figure \ref{fig:k_L2_norm_proj}.
        \begin{figure}
                \centering
                \begin{subfigure}{0.5\textwidth}
                        \centering
                        \includegraphics[width=1\linewidth]{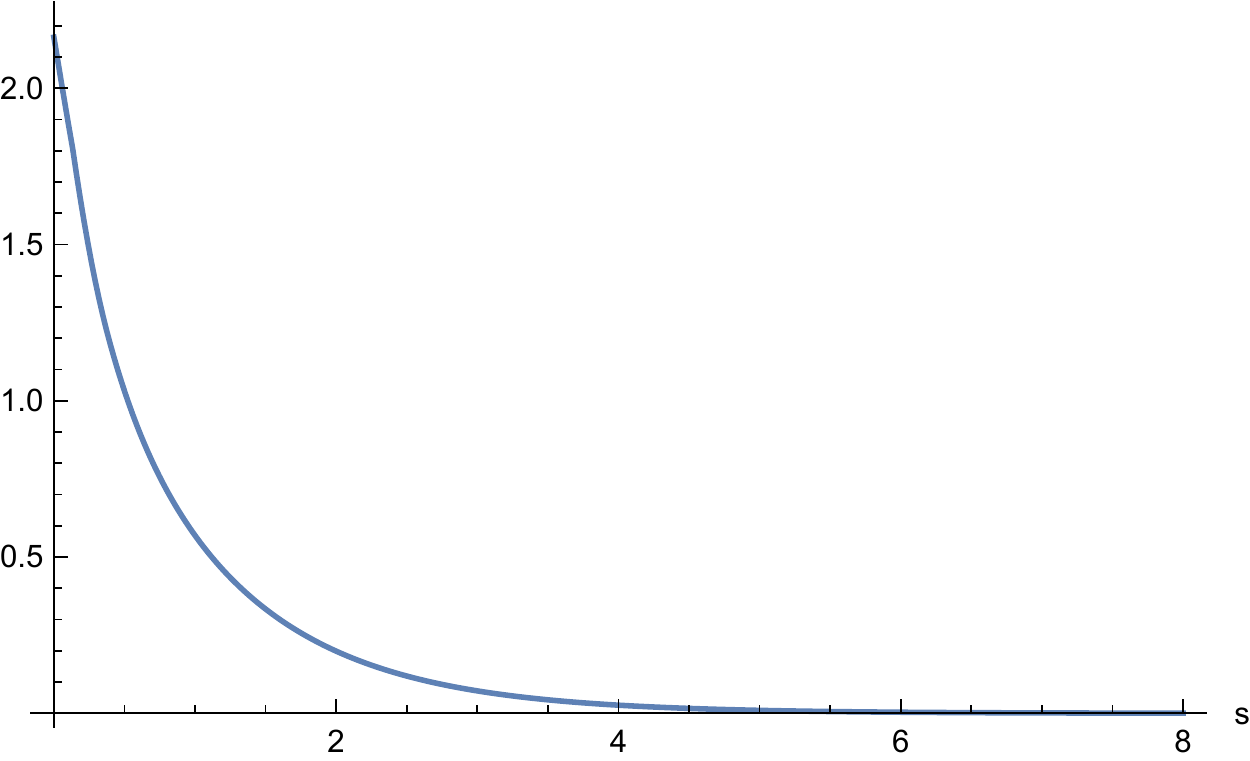}
                        \caption{Contraction of the norm.}
                        \label{fig:k_L2_norm_proj}
                \end{subfigure}
                \caption{Evolution of the norm, $\|(\theta,u)-(\bar \theta,\bar u)\|$, using the projection method.}
                \label{fig:k_L2_norm_proj}
        \end{figure}

\section{Initial-terminal value problems} 
\label{sec:mean_field_solution}
        The initial-terminal conditions in \eqref{itc} are 
         the key difficulty in the design of numerical methods for the time-dependent MFG, \eqref{MFM:cont_dyn}. 
        Here, we extend the strategy from the previous section to handle initial-terminal conditions. 
        We start with an arbitrary pair of functions, $(u(t,0),\theta(t,0))$, that satisfies \eqref{itc} and build a family 
        $(u(t,s),\theta(t,s))$, $s\geq 0$, that converges to a solution of \eqref{MFM:cont_dyn}-\eqref{itc} as $s\to \infty$,
        while preserving the boundary conditions for all $s\geq 0$.  
        \subsection{Representation of functionals in $H^1$} 
        \label{sub:representation_of_functionals_in_h_1}
                We begin by discussing the representation of linear functionals in $H^1$.
                Consider the Hilbert space $H^1_T=\{\phi \in H^1([0,T], \Rr^d): \phi(T)=0\}$.
                For $\theta, u\in H^1([0,T], \Rr^d)$, we consider the variational problem
                \begin{equation}
                \label{vp1}
                                \min_{\phi\in H^1_T} \int_0^T
                                \left[
                                \frac{1}{2}(|\phi|^2 + |\dot\phi|^2) + \phi \cdot \lb \theta_t - \sum_j \theta^j \alpha^*(\Delta_j u, \theta,j) \rb \right] dt
                 \end{equation}
                A minimizer, $\phi\in H^1_T$, of the preceding functional represents the linear functional
                \begin{equation*}
                        \eta\mapsto -\int_0^T\eta \cdot\lb \theta_t - \sum_j \theta^j \alpha^*(\Delta_j u, \theta,j) \rb dt
                \end{equation*}
                for $\eta\in H^1_T$,
                as an inner product in $H^1_T$; that is,
                                \begin{equation*}
                                        \int_0^T \lb \eta \cdot\phi + \dot \eta \cdot\dot \phi \rb dt = - \int_0^T \eta \cdot \lb \theta_t - \sum_j \theta^j \alpha^*(\Delta_j u, \theta,j) \rb dt
                                \end{equation*}
                for $\phi,\eta \in H^1_T$. 
                The last identity is simply the weak form of the Euler-Lagrange equation for \eqref{vp1},
                \begin{equation}\label{MFM:eq:ELeq_phi_cont_dyn}
                        - \ddot \phi + \phi = -\theta_t + \sum_j \theta^j \alpha^*(\Delta_j u, \theta,j),
                \end{equation}
                whose boundary conditions are $\phi(T)=0$ and $\dot \phi(0)=0$. For $\theta, u\in H^1([0,T], \Rr^d)$, we define
                \begin{equation}\label{Phi}
                        \Phi(\theta, u, t)=\phi(t).
                \end{equation}
                Next, 
             let $H^1_I=\{\psi \in H^1([0,T], \Rr^d): \psi(0)=0\}$.
                                For $\theta, u\in H^1([0,T], \Rr^d)$, we consider the variational problem
                \begin{equation}
                \label{vp2}
                                \min_{\psi\in H^1_I} \int_0^T \left[\frac{1}{2}(|\psi|^2 + |\dot\psi|^2) + \psi \cdot \lb u_t + h(\Delta_i u, \theta,i) \rb \right] dt.                \end{equation}
                The Euler-Lagrange equation for the preceding problem is 
                \begin{equation}\label{MFM:eq:ELeq_psi_cont_dyn}
                        - \ddot \psi + \psi = - u_t - h(\Delta_i u, \theta,i),
                \end{equation}
                with the boundary conditions $\psi(0)=0$ and $\dot \psi(T)=0$. Moreover, if $\psi\in H^1_I$ minimizes the functional in \eqref{vp2}, we have
                \begin{equation*}
                \int_0^T \lb\eta \cdot\psi + \dot \eta \cdot\dot \psi\rb dt = \int_0^T \eta \cdot\lb - u_t - h(\Delta_i u, \theta,i) \rb dt
                \end{equation*}
                for $\eta, \psi\in H^1_I$. 
                 For $\theta, u\in H^1([0,T], \Rr^d)$, we define
                \begin{equation}\label{Psi}
                        \Psi(\theta, u, t)=\psi(t).
                \end{equation}
        \subsection{Monotone deformation flow} 
        \label{sub:deformation_flow}
                Next, we introduce the monotone deformation flow,
                \begin{equation} \label{eq:in_term:aux_dyn}
                        \begin{cases}
                                u_s^i(t,s)              = \Phi^i(\theta(\cdot, s), u(\cdot, s), t)\\
                                \theta_s^i(t,s) = \Psi^i(\theta(\cdot, s), u(\cdot, s), t),                
                        \end{cases}
                \end{equation}
                where $\Phi$ and $\Psi$ are given in \eqref{Phi} and \eqref{Psi}.
                As we show in the next proposition, the previous flow is a contraction in $H^1$. 
                Moreover, if $(\theta, u)$ solve  \eqref{MFM:cont_dyn}-\eqref{itc}, we have
                \[
                	\Phi(\theta, u, t)=\Psi(\theta, u, t)=0. 
                \]
                Hence, solutions of  \eqref{MFM:cont_dyn}-\eqref{itc}  are fixed points for \eqref{eq:in_term:aux_dyn}. 

                Before stating the contraction property, we recall that the $H^1$-norm of a pair of functions is given by
                \[
                	\|(v, \eta)\|_{H^1}^2 = \int_0^T \lb |v|^2+|\dot v|^2+|\eta|^2+|\dot \eta|^2\rb dt
                \]
				for $v,\eta:[0,T]\to \Rr^d$. 
                \begin{pro}\label{contra1}
                        Let $(u, \theta)$ and $(\tilde u, \tilde \theta)$ solve \eqref{eq:in_term:aux_dyn}. Suppose $\theta, \tilde \theta \geq 0$. 
                        Then
                        \[
                        \frac{d}{ds}\|(u, \theta)-(\tilde u, \tilde \theta)\|^2_{H^1} \leq 0,
                        \]
                        with strict inequality if $(u, \theta)\neq (\tilde u, \tilde \theta)$.
                \end{pro}
                \begin{proof}
                	{We have
					\begin{equation*}
						\begin{split}
							\frac{1}{2} \frac{d}{ds} \int_0^T \left[ (u-\tilde u)^2 + (u-\tilde u)^2_t + (\theta-\tilde \theta)^2 + (\theta - \tilde \theta)^2_t \right]dt\\
							= \int_0^T \left[(u-\tilde u) (u-\tilde u)_s + (u-\tilde u)_t (u-\tilde u)_{ts}\right] dt\\
							+ \int_0^T \left[ (\theta-\tilde \theta)(\theta-\tilde \theta)_s  + (\theta-\tilde \theta)_t (\theta-\tilde \theta)_{ts}\right] dt.
						\end{split}
					\end{equation*}
					Using \eqref{eq:in_term:aux_dyn},  the term in the right-hand side of previous equality becomes
					\begin{equation*}
						\begin{split}
							\int_0^T \left[(u-\tilde u) (\phi-\tilde \phi) + (u-\tilde u)_t (\phi-\tilde \phi)_t \right] dt\\
							 + \int_0^T \left[  (\theta-\tilde \theta) (\psi - \tilde \psi) 
							 + (\theta-\tilde \theta)_t(\psi-\tilde \psi)_t\right] dt\\
							 = \int_0^T (u-\tilde u) (\phi-\tilde \phi) dt + \left.\left[(u-\tilde u)(\phi-\tilde \phi)_t \right]\right|_0^T \\
							 - \int_0^T (u-\tilde u)(\phi- \tilde \phi)_{tt} dt \\
							 + \int_0^T (\theta-\tilde \theta) (\psi-\tilde \psi)dt + \left.\left[(\theta-\tilde \theta)(\psi-\tilde \psi)_t \right]\right|_0^T \\
							 - \int_0^T (\theta-\tilde \theta)(\psi-\tilde \psi)_{tt} dt,
						\end{split}
					\end{equation*}
					where we used integration by parts in the last equality.}
                        Because $u(T)= \tilde u(T)$, $\theta(0) = \tilde \theta(0)$, $\phi_t(0) = \tilde \phi_t(0)$,  $\psi_t(T) = \tilde \psi_t(T)$, 
                        and using \eqref{MFM:eq:ELeq_phi_cont_dyn} and \eqref{MFM:eq:ELeq_psi_cont_dyn}, we obtain
                \begin{align} \label{itvp:norm_h1}
					\frac{1}{2} \frac{d}{ds} &\int_0^T (u-\tilde u)^2 + (u-\tilde u)^2_t + (\theta-\tilde \theta)^2 + (\theta - \tilde \theta)^2_t \nonumber\\
                	&= \int_0^T (u-\tilde u) \lb \sum \theta^j \alpha^*(\Delta_j u, \theta, j) - \sum \tilde \theta \alpha^*(\Delta_j \tilde u, \tilde \theta,j) \rb \nonumber \\
                    &- \int_0^T  (\theta- \tilde \theta) \lb h(\Delta_i u, \theta,i) - h(\Delta_i \tilde u,\tilde \theta,i ) \rb\\
                    & \leq \int_0^T -\gamma \|(\theta- \tilde \theta)(t) \|^2 - \sum_{i=1}^d \gamma_i (\theta^i + \tilde \theta^i)(t) \| (\Delta_i u - \Delta_i \tilde u)(t) \|^2 dt, \nonumber
                \end{align}
                due to Remark \ref{R1}.
                \end{proof}

        \subsection{Monotone discretization} 
        \label{sub:monotone_discretization}
        To build our numerical method, we begin by discretizing \eqref{eq:in_term:aux_dyn}. 
        We look for a time-discretization of
        \begin{equation*}
                A\left[\begin{array}{c}
                \theta\\
                u
                \end{array}\right]
                =
                \left[\begin{array}{c}
                -\theta_t + f(u,\theta)\\
                -u_t - h(u,\theta)
                \end{array}\right]
        \end{equation*}
        that preserves monotonicity, {where $f(u,\theta)=\sum_j \tilde \theta^j \alpha^*(\Delta_j \tilde u, \tilde \theta,j)$.}
                
        For Hamilton-Jacobi equations, implicit schemes have good stability properties. 
		Because the Hamilton-Jacobi equation in \eqref{eq:in_term:aux_dyn} is a terminal value problem, we discretize it using an explicit forward in time scheme (hence, implicit backward in time). 
		Then, to keep the adjoint structure of $A$ at the discrete level, we are then required to choose an implicit discretization forward in time for the first component of $A$.
        Usually, implicit schemes have the disadvantage of requiring the numerical solution of non-linear equations at each time-step. 
		Here, we discretize the operator $A$ globally, and we never need to solve implicit equations. 
        
        More concretely, we split $[0,T]$ into $N$ intervals of length  $\delta t=\frac T N$. 
        The vectors $\theta_n\in \Ss^d$ and $u_n\in \Rr^d$, $0\leq n\leq N$ approximate $\theta$ and $u$ at time $\frac{nT}N$. 
        We set $\mathcal{M}_N=(\Ss^d\times \Rr^{d})^{N+1}$ and define
        \begin{equation} \label{disc_probl}
        A^N\left[\begin{array}{c}
        \theta\\
        u
        \end{array}\right]_n
        =
        \left[\begin{array}{c}
         -\frac{\theta_{n+1}^i-\theta_n^i}{\delta t}+f(u_{n+1}^i,\theta_{n+1}^i) + k_n\\ 
         -\frac{u_{n+1}^i-u_n^i}{\delta t}-h(u_{n}^i,\theta_{n}^i)
        \end{array}\right],
        \end{equation}
        where 
        \begin{equation*}
                k_n(s) = -\frac{1}{d} \sum_{i=1}^d \lb -\frac{\delta \theta^i_n}{\delta t} + f(u_{n+1}^i,\theta_{n+1}^i)\rb
        \end{equation*}
                and $\delta \theta^i_n = \theta_{n+1}^i-\theta_n^i$.
Next, we show that $A^N$ is a monotone operator in the convex subset of vectors in $\mathcal{M}$ that satisfy the initial-terminal conditions in \eqref{itc}.

\begin{pro}
$A^N$ is monotone in the convex subset 
 $\mathcal{M}_N$ of all $(\theta, u)\in (\Ss^d\times \Rr^d)^{N+1}$
such that $\theta_0=\bar \theta_0$ and $u_{N}=\bar u_T$.  Moreover, we have the inequality
\begin{align*}
&\left\langle A^N\left[\begin{array}{c}
\theta\\
u
\end{array}\right]-A^N\left[\begin{array}{c}
\tilde\theta\\
\tilde u
\end{array}\right],\left[\begin{array}{c}
\theta\\
u
\end{array}\right]-\left[\begin{array}{c}
\tilde\theta\\
\tilde u
\end{array}\right]\right\rangle
\\
&\leq \sum_{n=1}^{N-1} \lb -\gamma \|(\theta- \tilde \theta)(t) \|^2 - \sum_{i=1}^d \gamma_i (\theta^i + \tilde \theta^i)(t) \| (\Delta_i u - \Delta_i \tilde u)(t) \|^2 \rb.
\end{align*}    
\end{pro}       
\begin{proof}
We begin by computing   
\begin{align*}
&\left\langle A^N\left[\begin{array}{c}
\theta\\
u
\end{array}\right]-A^N\left[\begin{array}{c}
\tilde\theta\\
\tilde u
\end{array}\right],\left[\begin{array}{c}
\theta\\
u
\end{array}\right]-\left[\begin{array}{c}
\tilde\theta\\
\tilde u
\end{array}\right]\right\rangle\\
                &=\sum_{n=0}^{N-1} (\theta_n - \tilde \theta_n) \left(-\frac{u_{n+1}- u_{n}}{\delta t} - h(u_{n},\theta_{n}) 
                + \frac{\tilde u_{n+1}-\tilde u_{n}}{\delta t} + h(\tilde u_{n},\tilde \theta_{n}) \right)\\
                &+ (u_{n+1}-\tilde u_{n+1}) \Big( -\frac{\theta_{n+1}-\theta_{n}}{\delta t} + f(u_{n+1},\theta_{n+1}) + k_{n}\\
                &+ \frac{\tilde \theta_{n+1}-\tilde\theta_{n}}{\delta t} - f(\tilde u_{n+1},\tilde \theta_{n+1}) - \tilde k_n \Big).
        \end{align*}
        Developing the sums and relabeling the indices,
        the preceding expression becomes
        \begin{align*}
                &\sum_{n=1}^{N-1} (\theta_n - \tilde \theta_n) \left(-\frac{u_{n+1}- u_{n}}{\delta t} - h(u_{n},\theta_{n}) 
                        + \frac{\tilde u_{n+1}-\tilde u_{n}}{\delta t} + h(\tilde u_{n},\tilde \theta_{n}) \right)\\
                        &+ (\theta_0-\tilde \theta_0) \left(- \frac{u_1-u_0}{\delta t} - h(u_0,\theta_0) + \frac{\tilde u_1 -\tilde u_0}{\delta t} +h(\tilde u_0,\theta_0) \right)\\
            &+\sum_{n=1}^{N-1} (u_n-\tilde u_n) \lb -\frac{\theta_{n}-\theta_{n-1}}{\delta t} + f(u_{n},\theta_{n})
                                + \frac{\tilde \theta_{n}-\tilde\theta_{n-1}}{\delta t} - f(\tilde u_{n},\tilde \theta_{n})   \rb\\
                        &+(u_{N}-\tilde u_{N})\lb-\frac{\theta_N-\theta_{N-1}}{\delta t} +f(u_N,\theta_N) + \frac{\tilde \theta_N-\tilde \theta_{N-1}}{\delta t}- f(\tilde u_N,\tilde \theta_N) \rb.
        \end{align*}
        The second and last lines above are zero since $\theta_0=\tilde \theta_0=\bar \theta_0$ and $u_{N}=\tilde u_{N}=\bar u_T$.
        Using Remark \ref{R1}, we obtain
        \begin{align*}
                                &\left\langle A\left[\begin{array}{c}
                                \theta\\
                                u
                                \end{array}\right]-A\left[\begin{array}{c}
                                \tilde\theta\\
                                \tilde u
                                \end{array}\right],\left[\begin{array}{c}
                                \theta\\
                                u
                                \end{array}\right]-\left[\begin{array}{c}
                                \tilde\theta\\
                                \tilde u
                                \end{array}\right]\right\rangle \\
                &\leq \sum_{n=1}^{N-1} \lb -\gamma \|(\theta- \tilde \theta)(t) \|^2 - \sum_{i=1}^d \gamma_i (\theta^i + \tilde \theta^i)(t) \| (\Delta_i u - \Delta_i \tilde u)(t) \|^2 \rb \\
                &- \sum_{n=1}^{N-1} (\theta_n-\tilde\theta_n) \lb \frac{u_{n+1}-\tilde u_{n+1}}{\delta t} - \frac{u_n-\tilde u_n}{\delta t} \rb\\
                &- \sum_{n=1}^{N-1} (u_n-\tilde u_n) \lb\frac{\theta_n -\tilde \theta_n}{\delta t}-\frac{\theta_{n-1}-\tilde \theta_{n-1}}{\delta t}\rb.
        \end{align*}
        We now show that the last two lines add to zero.
        Let $a_n = \theta_n-\tilde \theta_n$ and $b_n = u_n-\tilde u_n$.
                Accordingly, we have
        \begin{align*}
                &-\frac{1}{\delta t} \sum_{n=1}^{N-1} a_n(b_{n+1}-b_n) - \frac{1}{\delta t} \sum_{n=1}^{N-1} b_n (a_n-a_{n-1}) \\
                &=\frac{1}{\delta t} \sum_{n=1}^{N-1} b_{n+1}(a_{n+1}-a_n) - (b_N a_N - b_1 a_1) - \frac{1}{\delta t} \sum_{n=0}^{N-2} b_{n+1} (a_{n+1}-a_n)\\
                &=\frac{1}{\delta t}(b_1a_0-b_Na_{N-1}) = 0,
        \end{align*}
        where we summed by parts the first member and relabeled the index $n$ in the last member of the first line.
        The last equality follows from the assumption in the statement, $a_0=\theta_0-\tilde \theta_0 = 0$ and $b_N = u_N-\tilde u_N = 0$.     
\end{proof}     
	 Using the techniques in \cite{AFG}, we prove the convergence of  the solutions of the discretized problem as $\delta t \to 0$. As usual, we discretize the time interval, $[0,T]$, into $N+1$ equispaced points. 
	 \begin{pro}
		 Let $(\theta^N, u^N)\in \Mm_N$, be a solution of 
		 \begin{equation*}
			 A^N\left[
			 \begin{array}{c}
				 \theta^N\\
				 u^N
			 \end{array}
			 \right]_n 
			 = 
			 \left[
			 \begin{array}{c}
				 0\\
                 0
             \end{array}
			 \right]
		\end{equation*}
		satisfying the initial-terminal conditions in \eqref{itc}.
		Suppose $u^N$ is uniformly bounded. Consider the step functions, $\bar u^N$, $\bar \theta^N$, taking the values $\bar u^{N i}_n \in \Rr$ and $\bar{\theta}_n^{N i} \in \Ss$ in $[\frac{(n-1)T}{N},\frac{n T}{N}]$, 
		with $0\leq n \leq N$, for $i\in I_d$, respectively.
		Then, extracting a subsequence if necessary,  $\bar u^{N i} \rightharpoonup \bar u^i$ and $\bar \theta^{N i} \rightharpoonup \theta^i$, weakly-* in $L^\infty$ for $i \in I_d$. 
		Furthermore, $(\bar u, \bar \theta)$ is a weak solution of \eqref{MFM:cont_dyn}.
    \end{pro}
    \begin{proof}
    	Because $u^N$ is bounded by hypothesis and $\theta^N$ is bounded since it is a probability measure, the weak-* convergence in $L^\infty$ is immediate. Hence,  there exist $\bar u^i \in L^\infty([0,T])$ and $\theta^i \in L^\infty([0,T])$ as claimed. 
		
    	Let $\tilde u^{i}, \tilde \theta^{i} \in C^{\infty}([0,T])$, with $\tilde \theta^{i} \geq 0$ for all $i\in I_d$, and $\sum_{i \in I_d} \tilde \theta^i =1$. Suppose further that $\tilde u^{i}, \tilde \theta^{i}$ satisfy the boundary conditions in \eqref{itc}.
		Let $\tilde u_n^{N} = \tilde u \lb \frac{n}{N}T \rb$, $\tilde \theta_n^{N} = \tilde \theta \lb \frac{n}{N}T \rb$ be the vectors whose components are $\tilde u_n^{N i}$ and
		$\tilde \theta_n^{N i}$, respectively.
		By the monotonicity of $A^N$, we have
		\begin{align*}
			0 &\leq \left\langle A^N\left[\begin{array}{c}
											\tilde \theta^N \\
                                			\tilde u^N
                                		 \end{array}
									 \right],
								\left[\begin{array}{c}
                                	\tilde \theta^N\\
                                	\tilde u^N
                                \end{array}\right]
								-\left[\begin{array}{c}
                                			\theta^N\\
                                			u^N
                                		\end{array}
								 \right]
					\right\rangle \\
				& = O \lb \frac{1}{N} \rb 
				  +\left\langle
				  		A^N \left[\begin{array}{c}
									\tilde\theta\\
                                	\tilde u
                                \end{array}\right],
						\left[\begin{array}{c}
									\tilde \theta\\
						            \tilde u
						      \end{array}
						\right]
						- \left[\begin{array}{c}
									\bar \theta^N\\
						            \bar u^N
						   		 \end{array}
						\right]
				   \right \rangle,
		\end{align*}
		taking the limit $N\to \infty$ gives the result. 
    \end{proof}

\subsection{Monotone discretization for the $H^1$ projections} 
\label{sub:monotone_discretization_for_the_h_1_projections}
        Next, we discuss the representation of linear functionals for the discrete problem. For that, proceeding as in Section \ref{sub:representation_of_functionals_in_h_1}, 
         we compute the optimality conditions of the discretized versions of \eqref{vp1} and  \eqref{vp2}.
        
        Fix $(u,\theta)\in \mathcal{M}_N$ and consider the following discrete analog to \eqref{vp1}:
    \begin{equation*}
                \min_{\phi\in \tilde H^1_T} \delta t \sum_{n=1}^{N} \frac{1}{2} \lb\phi_n^2 + \lb\frac{\delta \phi_{n-1}}{\delta t}\rb^2\rb 
            + \phi_n \lb \frac{\delta \theta_{n-1}}{\delta t} - f(u_{n},\theta_{n}) \rb,
     \end{equation*}
     where $\delta g_n = g_{n+1}-g_n$, and  $\tilde H^1_T =\{\phi=(\phi_0, \hdots, \phi_N)\in (\Rr^d)^{(N+1)}:\phi_N=0\}$.
     The corresponding optimality conditions (the discrete Euler-Lagrange equation) is 
     \begin{equation}\label{eq:disc_el_phi}
     -\frac{\delta(\delta \phi_{n-1})}{(\delta t)^2} + \phi_n = -\frac{\delta\theta_{n-1}}{\delta t} + f(u_{n},\theta_{n}),
     \end{equation}
     for $n=1,\ldots,N-1$, coupled with the boundary conditions $\phi_N = 0$ and $\phi_{1}=\phi_0$.
         
         A minimizer of the problem above represents the following discrete linear functional
         \begin{equation*}
                \eta \mapsto - \sum_{n=1}^N \eta_n \cdot \lb \frac{\delta \phi_{n-1}}{\delta t} - f(u_n,\theta_n) \rb \delta t
         \end{equation*}
     as an inner product in $\tilde H^1_T$
         \begin{equation*}
                \sum_{n=1}^N \lb \eta_n \cdot \phi_n \delta t + \frac{1}{\delta t}\delta \eta_{n-1}\cdot  \delta \phi_{n-1} \rb  = - \sum_{n=1}^N\eta_n \cdot\lb \frac{\delta \phi_{n-1}}{\delta t} - f(u_n,\theta_n) \rb \delta t.
         \end{equation*}
         For $(\theta_n,u_n) \in \mathcal{M}_N$, we define
          \begin{equation}
          \label{phid}
                \Phi(\theta_n,u_n) = \phi_n.
          \end{equation}
                                       
         We now examine a second discrete variational problem corresponding to \eqref{vp2}. 
         For $(u,\theta)\in \mathcal{M}_N$, consider
     \begin{equation*}
                \min_{\psi\in \tilde H^1_I} \delta t \sum_{n=0}^{N-1} \frac{1}{2} \lb\psi_n^2 + \lb\frac{\delta\psi_n}{\delta t}\rb^2\rb 
            + \psi_n \lb \frac{\delta u_n}{\delta t} + h(u_n, \theta_n) \rb,   
      \end{equation*}
      where $\tilde H^1_I =\{\psi=(\psi_0, \hdots, \psi_N)\in (\Rr^d)^{(N+1)}:\psi_0=0\}$.
          
      The discrete Euler-Lagrange equation is 
      \begin{equation}\label{eq:disc_el_psi}
        -\frac{\delta(\delta\psi_{n-1})}{(\delta t)^2} + \psi_n = - \frac{\delta u_n}{\delta t} - h(u_n, \theta_n)
      \end{equation}
      for $n=1,\ldots,N-1$, together with the conditions $\psi_0=0$ and $\psi_N=\psi_{N-1}$.
          
          From the Euler-Lagrange equation, we obtain the following representation formula in the Hilbert space $\{\psi \in H_n^1(\{0,\ldots,N\}): \psi_0 = 0\}$:  
          \begin{equation*}
                \sum_{n=0}^{N-1} \lb \eta \cdot \psi + \delta \eta  \cdot \delta \psi \rb \delta t =  \sum_{0}^N \eta \cdot \lb - \frac{\delta u_n}{\delta t}-h(u_n,\theta_n)  \rb \delta t.
          \end{equation*}
     Finally, we define
          \begin{equation}
          \label{psid}
                \Psi(\theta_n,u_n) = \psi_n,
          \end{equation}
          for $(u,\theta)\in \mathcal{M}_N$.
        \begin{pro} \label{prop:monot}
                Let $\Phi$ and $\Psi$ be given by \eqref{phid} and \eqref{psid}. Consider the following operator:
        \begin{equation}
        \label{qaop}
                Q_A\left[\begin{array}{c}
            \theta\\
            u   
            \end{array}\right]
                                =
                                \left[\begin{array}{c}
                                \Phi\\
                                \Psi
                                \end{array}\right].
        \end{equation}
           Let $\Mm_N^{\bar \theta_0,\bar u_T}$ be the set of all
             $(\theta, u) \in \Mm_N$ that satisfy the initial condition $\theta_0=\bar \theta_0$ and the terminal condition $u_N=\bar u_T$.      
                Then, $Q_A$ is monotone with respect to the discrete $H^1_N$ inner product corresponding to the norm
                \begin{equation}
                \label{h1nn}
                        \|(\eta, \nu )\|_{H^1_N}^2 = \sum_{n=0}^{N-1} |\eta_n|^2 + |\delta \eta_n|^2 + |\nu_n|^2 + |\delta \nu_n|^2.
                \end{equation}
        \end{pro}               
                                 
        \begin{proof}
                Let $(u,\theta)\in \Mm_N^{\bar \theta_0,\bar u_T}$ and $(\tilde u, \tilde \theta)\in \Mm_N^{\bar \theta_0,\bar u_T}$. Let $\phi, \tilde \phi$ and $\psi, \tilde\psi$ be given by \eqref{phid} and \eqref{psid}. 
                We begin computing
                \begin{align}
        \label{eq:H1norm_disc}\notag
                        &\left\langle Q_A\left[\begin{array}{c}
                        \theta\\
                        u
                        \end{array}\right]- Q_A\left[\begin{array}{c}
                        \tilde\theta\\
                        \tilde u
                        \end{array}\right],\left[\begin{array}{c}
                        \theta\\
                        u
                        \end{array}\right]-\left[\begin{array}{c}
                        \tilde\theta\\
                        \tilde u
                        \end{array}\right]\right\rangle_{H^1_N}
                        \\
            &=\sum_{n=0}^{N-1} (\theta_n-\tilde \theta_n)(\psi_n - \tilde\psi_n) 
                +\frac{\delta (\theta_n - {\tilde{\theta}}_n)}{\delta t} \frac{\delta (\psi_n - \tilde{\psi}_n)}{\delta t} \nonumber\\
            &\quad+ (u_n-\tilde u_n)(\phi_n-\tilde \phi_n) + \frac{\delta (u_n -{\tilde u}_n)}{\delta t}\frac{\delta({\phi}_n-{\tilde \phi}_n)}{\delta t} \nonumber\\
            &=\sum_{n=0}^{N-1} (\theta_n-\tilde \theta_n)(\psi_n - \tilde\psi_n) + (u_n-\tilde u_n)(\phi_n-\tilde \phi_n) \nonumber\\
            &\quad+\frac{1}{\delta t}\sum_{n=0}^{N-1}\lb\frac{\theta_{n+1}-\theta_n}{\delta t} 
                - \frac{\tilde\theta_{n+1}-\tilde\theta_n}{\delta t}\rb  (\delta{\psi}_n-\delta{\tilde \psi}_n)\nonumber\\
            &\quad+\frac{1}{\delta t}\sum_{n=0}^{N-1}\lb \frac{u_{n+1}-u_n}{\delta t} 
                - \frac{\tilde u_{n+1}-\tilde u_n}{\delta t} \rb (\delta{\phi}_n - \delta{\tilde \phi}_n).
        \end{align}
                Reorganizing, we see that the previous two lines are equal to
                \begin{align}\label{eq:mdisc_h1}
                        \frac{1}{(\delta t)^2} \sum_{n=0}^{N-1} \lsb(\theta_{n+1}-\tilde\theta_{n+1})-(\theta_n-\tilde\theta_n)\rsb (\delta{\psi}_n-\delta\tilde{\psi}_n)\nonumber\\
                                \quad+\lsb(u_{n+1}-\tilde u_{n+1})-(u_n-\tilde u_n)\rsb (\delta{\phi}_n-\delta{\tilde \phi}_n).
                \end{align}
                Using the notation
                \begin{equation*}
                        a_n=\theta_n-\tilde \theta_n, \quad b_n=\delta{\psi}_n-\delta{\tilde\psi}_n, \quad c_n=u_n-\tilde u_n, \text{ and } d_n = \delta{\phi}_n-\delta{\tilde\phi}_n,
                \end{equation*}
                we write \eqref{eq:mdisc_h1} multiplied by $(\delta t)^2$ as
                \begin{align}\label{eq:aux_disc}
                        \sum_{n=0}^{N-1} b_n \delta a_n + d_n \delta c_n &= b_{N-1}\delta a_{N-1} + d_{N-1}\delta c_{N-1} +\sum_{n=0}^{N-2} b_n \delta a_n + d_n \delta c_n \nonumber\\
                        & = b_{N-1}\delta a_{N-1} + d_{N-1}\delta c_{N-1} \nonumber\\
                        & + a_{N-1}b_{N-1} - a_0 b_0 -\sum_{n=0}^{N-2} a_{n+1} \delta b_n \nonumber\\
                        & + c_{N-1}d_{N-1} - c_0d_0 - \sum_{n=0}^{N-2} c_{n+1} \delta d_n,
                \end{align}
                where we used summation by parts in the last equality.
                Because  $\psi_N=\psi_{N-1}$, we have that $b_{N-1}=0$. 
                Moreover, since $\theta_0=\tilde\theta_0$, we have $a_0=0$, and $\phi_1=\phi_0$ implies
                $d_0=0$.
                                Thus, we further have
                \begin{align*}
                        d_{N-1}\delta c_{N-1} &= d_{N-1}\lb u_N-\tilde u_N - (u_{N-1}-\tilde u_{N-1}) \rb \\
                                &= -d_{N-1}(u_{N-1}-\tilde u_{N-1})\\
                                &= -c_{N-1}d_{N-1},
                \end{align*}
                where we used the terminal condition $u_N=\tilde u_N$.
                        According to these identities, \eqref{eq:aux_disc} becomes
                \begin{align*}
                        \sum_{n=0}^{N-1} b_n \delta a_n + d_n \delta c_n &= -\sum_{n=0}^{N-2} a_{n+1} \delta b_n + c_{n+1} \delta d_n.
                \end{align*}
                Therefore, \eqref{eq:mdisc_h1} can be written as
                \begin{equation}\label{eq:mdisc_h1:eq2}
                        -\sum_{n=0}^{N-2}\frac{\theta_{n+1}-\tilde\theta_{n+1}}{(\delta t)^2} \lb \delta^2 \psi_n - \delta^2 \tilde \psi_n \rb
                        -\sum_{n=0}^{N-2}\frac{u_{n+1}-\tilde u_{n+1}}{(\delta t)^2} \lb \delta^2 \phi_n - \delta^2 \tilde \phi_n \rb.
                \end{equation}
                Shifting the index $n+1$ into $n$ in \eqref{eq:mdisc_h1:eq2}, we obtain
                \begin{align*}
                        -\sum_{n=1}^{N-1}\frac{\theta_{n}-\tilde\theta_{n}}{(\delta t)^2} \lb \delta^2 \psi_{n-1} - \delta^2 \tilde \psi_{n-1} \rb
                        -\sum_{n=1}^{N-1}\frac{u_{n}-\tilde u_{n}}{(\delta t)^2} \lb \delta^2 \phi_{n-1} - \delta^2 \tilde \phi_{n-1} \rb.
                \end{align*}
                Using the Euler-Lagrange equations \eqref{eq:disc_el_phi} and \eqref{eq:disc_el_psi} in the preceding expression,  yields
                \begin{align*}
                        &-\sum_{n=1}^{N-1} (\theta_n-\tilde \theta_n) \lb \psi_n + \frac{u_{n+1}-u_n}{\delta t} + h(u_n,\theta_n) 
                                - \tilde\psi_n - \frac{\tilde u_{n+1}-\tilde u_n}{\delta t} - h(\tilde u_n,\tilde \theta_n)\rb \\
                        &-\sum_{n=1}^{N-1} (u_n-\tilde u_n) \lb \phi_n + \frac{\theta_n - \theta_{n-1}}{\delta t} - f(u_n,\theta_n) 
                                - \tilde \phi_n - \frac{\tilde\theta_n - \tilde \theta_{n-1}}{\delta t} + f(\tilde u_n, \tilde \theta_n)\rb.
                \end{align*}
                Finally, plugging the previous result into \eqref{eq:H1norm_disc}, we obtain
                \begin{align*}
                        &-\sum_{N=1}^{N-1} (\theta_n-\tilde \theta_n) \lb \frac{u_{n+1}-\tilde u_{n+1}}{\delta t} - \frac{ u_{n}- \tilde u_n}{\delta t} 
                                + h(u_n,\theta_n) - h(\tilde u_n, \tilde \theta_n) \rb \\
                        &-\sum_{n=1}^{N-1} (u_n-\tilde u_n) \lb \frac{\theta_n - \tilde \theta_{n}}{\delta t} -\frac{\theta_{n-1}- \tilde\theta_{n-1}}{\delta t} 
                                - f(u_n,\theta_n) + f(\tilde u_n,\tilde \theta_n) \rb\\
                        &\leq -\gamma \|(\theta- \tilde \theta)(s)\|^2 -\sum_{i=1}^d \gamma_i(\theta^i+\tilde\theta^i)(s) \|(\Delta_i u-\Delta_i \tilde u)(s)\| ^2,
                \end{align*}
               using Remark \ref{R1} and arguing as at the end of Subsection \ref{sub:monotone_discretization}.
        \end{proof}             
\subsection{Projection algorithm} 
\label{sub:projection_algorithm}
        As shown in Section \ref{sec:stationary_problem}, the monotone flow may not keep $\theta$ positive. 
        Thus, to preserve probabilities and prevent $\theta$ to take negative values, we define a projection operator through the following optimization problem. 
         Given $(\eta, w)\in \Mm_N$, we solve
         \begin{equation}\label{eq:pos}
         \begin{cases}
         \min_{\lambda^i_n} \sum_{i=1}^d (\eta^i_n - \lambda^i_n)^2 \\
         \sum_{i=1}^d \lambda^i_n = 1, \quad \lambda^i_n \geq 0
         \end{cases}
         \end{equation}
         for $n\in\{0,\ldots, N\}$. 
         Then, we set
         \[      P
         \left[\begin{array}{c}
         \eta\\
         w
         \end{array}\right]_n=
         \left[\begin{array}{c}
         \lambda_n\\
         w_n
         \end{array}\right]
         \]
         for $0\leq n \leq N$. We note that if $\eta_n$ is a probability, then $\lambda_n=\eta_n$. Moreover, $P$ is a contraction.
        
        Now, we introduce 
        the following iterative scheme:
        \begin{equation}\label{eq:prj}
                w_{k+1} = P\left[w_k - \upsilon  Q_A [w_k]\right],
        \end{equation}
        where $w_k = (\theta_k, u_k)$, $Q_A$ is defined in \eqref{qaop}, and $\upsilon>0$ is the step size. 
       \begin{pro}
       	For $\upsilon$ small enough,
       the map \eqref{eq:prj} is a contraction.
       \end{pro}
       \begin{proof}
       	The operator $E_{\upsilon}$ is a contraction, because $Q_A$ is a monotone Lipschitz map, see Proposition \ref{prop:monot}.
       \end{proof}

        \begin{pro} 
                Let $(\bar \theta, \bar u)\in \Mm_N$ solve 
            \begin{equation*}
                A^N\left[\begin{array}{c}
                \theta\\
                u
                \end{array}\right]=
                \left[\begin{array}{c}
                        0\\
                        0
                \end{array}\right],
            \end{equation*}
            with $u_N=\bar u_T$ and $\theta_0=\bar \theta_0$.       
            Then, $(\bar \theta, \bar u)$ is a fixed point of \eqref{eq:prj}.
                
            Conversely, let  $(\tilde \theta, \tilde  u)\in \Mm_N$  be a fixed point of \eqref{eq:prj}
            with $\tilde{\theta}>0$. 
            Then, there exists a weak solution to \eqref{MFM:cont_dyn}, $(\bar \theta, \bar u)$ with
            $\bar \theta=\tilde{\theta}$ and $\bar u$ given by
            \begin{equation}
            \label{recovery}
            \frac{\delta \bar u^i_n}{\delta t}= - h(\Delta_i \tilde u,\theta,i),
            \end{equation}         
            with $\bar u_N=\bar u_T$.
        \end{pro}
        \begin{proof}
                        The first claim of the proposition is immediate from the definition of $Q_A$. 
            To prove the second part, let $(\tilde \theta,\tilde u)\in \Mm_N$ be a fixed point of \eqref{eq:prj}. 
            For all $n\in \{0,\ldots,N\}$ and $i \in I_d$, we have
            \begin{equation*}
                \tilde u^i_n = \tilde u^i_n + \upsilon \phi_n(\tilde u_n, \tilde \theta_n).
            \end{equation*}
            Therefore, $\phi_n(\tilde u_n, \tilde \theta_n)=0$. Hence, from the uniqueness
            of the solution of \eqref{eq:disc_el_phi}, we conclude that 
            \[
 -\frac{\delta\tilde \theta_{n-1}}{\delta t} + f(\tilde u_{n},\tilde \theta_{n})=0           
            \]

            Furthermore, for $\tilde \theta_n^i = \lambda_n^i$, where $\lambda_n^i$ solves \eqref{eq:pos}, we have
                \begin{equation*}
                        \begin{split}
                        \tilde \theta_n^i &= P \left[ \tilde \theta_n^i - \upsilon \psi_n(\tilde u_n, \tilde \theta_n) \right] \\
                                        &= \lb \tilde \theta_n^i - \upsilon \psi_n(\tilde u_n, \tilde \theta_n) + \upsilon \kappa_n \rb^+
                    \end{split}
                \end{equation*}
            For some $\kappa_n\geq 0$. If $\tilde \theta^i_n > 0$, $\psi_n(\tilde u_n, \tilde \theta_n)=\kappa_n$. 
        Otherwise,
        $\psi_n(\tilde u_n, \tilde \theta_n) \geq \kappa_n$.
        
        If $\tilde \theta^i_n > 0$, using the fact that $\psi$ solves \eqref{eq:disc_el_psi}, we gather
          \begin{equation*}
          \frac{\delta \tilde u^i_n}{\delta t} - \frac{1}{d-1}\sum_{j \neq i} \frac{\delta \tilde u^j_n}{\delta t} = \frac{1}{d-1}\sum_{j \neq i} h(\Delta_j \tilde u_n,\theta,j) - h(\Delta_i \tilde u_n,\theta,i).
          \end{equation*}
          Now, we define $\bar u$ as in the statement of the proposition.  
          A simple computation gives 
          \[
           \frac{\delta \bar u^i_n}{\delta t}- \frac{\delta \bar u^j_n}{\delta t}=
           \frac{\delta \tilde u^i_n}{\delta t}- \frac{\delta \tilde u^j_n}{\delta t}.
          \]
          Hence, $\Delta_j \bar u_n=\Delta_j \tilde u_n$. Consequently, 
            \begin{equation*}
            \frac{\delta \bar u^i_n}{\delta t}= - h(\Delta_i \bar u,\theta,i).
            \end{equation*}
Thus, $(\bar \theta, \bar u)$ solves \eqref{MFM:cont_dyn}.            
        \end{proof}

\subsection{Numerical examples} 
\label{sec:numerical_implementation}

    Finally, we present numerical simulations for the  time-dependent paradigm-shift problem.
    As explained before, we discretize the time variable, $t \in [0,T]$, into $N$ intervals of length $\delta t=\frac{T}{N}$. 
    We then have $N$ equations for each state. Because $d=2$, 
    this system consists of $4N$ equations evolving according to 
    \eqref{eq:prj}.

        To compute approximate solutions to \eqref{sys:td_sol_hj}-\eqref{sys:td_sol_kolm}, we use the projection algorithm, \eqref{eq:prj}, with $N=450$.
    We first consider a case where the analytical solution can be computed explicitly. We choose 
    $\theta^1=\theta^2=\frac 1 2$. Thus, from \eqref{sys:td_sol_hj}, it follows that  $u^1=u^2$ are affine functions of $t$ with $u^1_t=u^2_t-\frac{1}{2}$. 
    Our results are depicted in Figures \ref{AA1}, \ref{AA2}, and \ref{AA3}.
    In Figure \ref{AA1}, for $t\in[0,T]$, $T=8$, we plot the initial guess ($s=0$) for $\theta$ and $u$, and the analytical solution.
    In Figure \ref{AA2}, we see the evolution of density of players and the value functions for $s\in[0,20]$. The final results, $s=100$, are shown in Figure \ref{AA3}. 
    Finally, in Figure \ref{AA4}, we show the evolution of the $H^1$ norm of the difference between the analytical, $(\tilde u, \tilde \theta)$, and computed solution, $(u,\theta)$.        The norm $\|(\tilde u, \tilde \theta) - (u,\theta) \|_{H^1([0,T])}^2(s)$ is computed as 
    \begin{equation*}
    \sum_{j=0}^{N-1} \sum_{i=1}^2 \delta t \lb|\tilde u^i_j - u^i_j |^2 + |\dot{\tilde{u}}^i_j - \dot u^i_j |^2 
    + |\tilde{\theta}^i_j -\theta^i_j|^2 + |\dot{\tilde{\theta}}^i_j -\dot{\theta}^i_j|^2 \rb(s)
    \end{equation*}
    for $s\geq 0$, where $v^i_j=v^i(t_j,s)$ and $\delta t$ is  size of the time-discretization step.
        
    \begin{figure}
                         \centering
                        \begin{subfigure}{0.5\textwidth}
                                         \centering
                \includegraphics[width=1\linewidth]{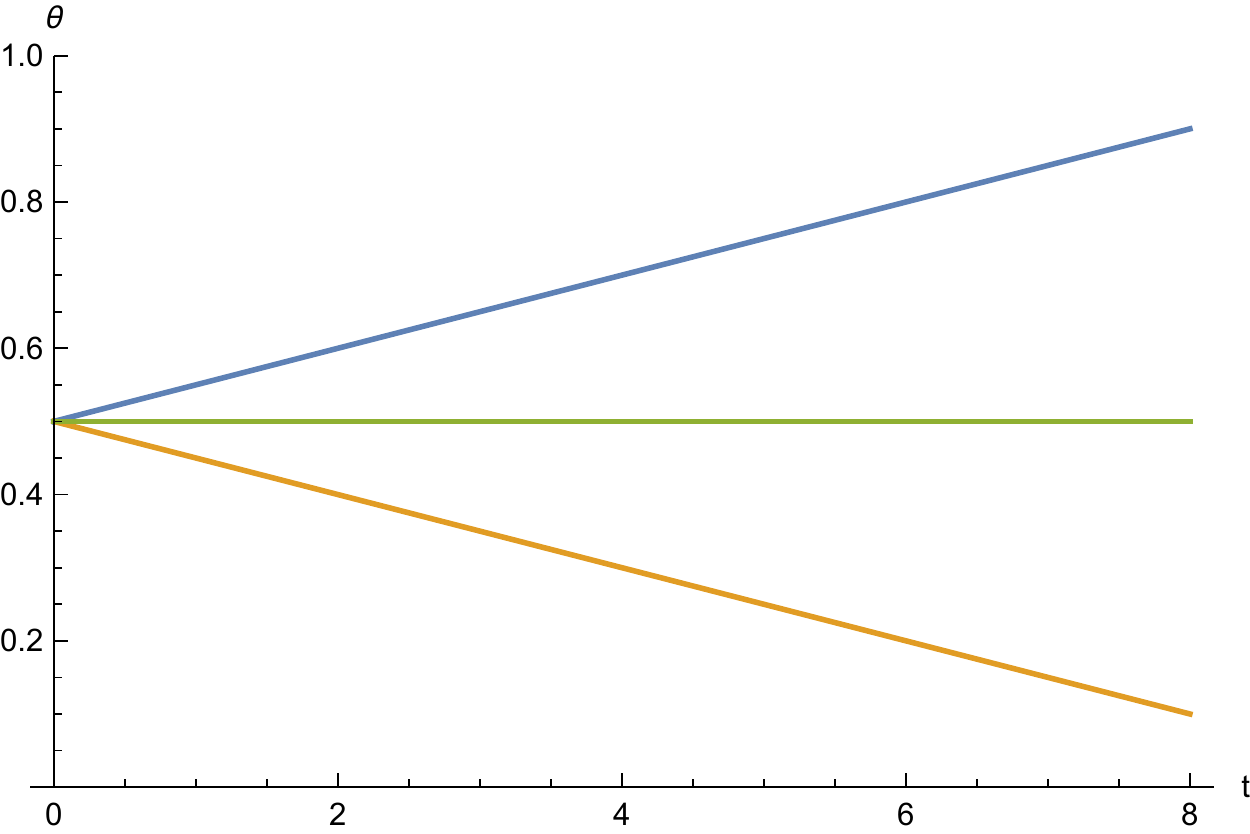}
                \caption{Initial condition $\theta(\cdot,s=0)$ vs exact solution.}
               \label{fig:it_cond}
                        \end{subfigure}
            \begin{subfigure}{0.5\textwidth}
            \centering
                \includegraphics[width=1\linewidth]{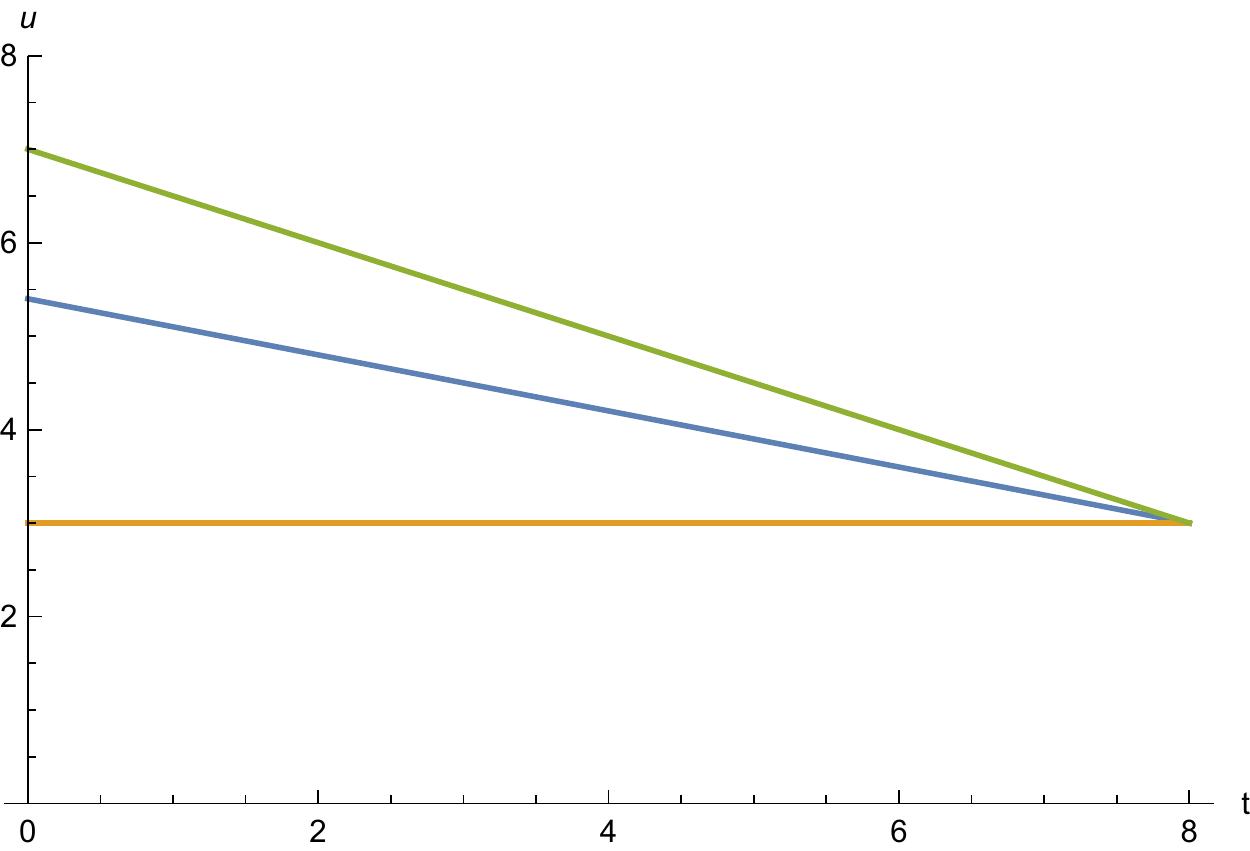}
                \caption{Initial condition $u(\cdot,s=0)$.}
            \end{subfigure}
                \caption{The blue lines correspond to the initial values ($s=0$) for state 1, $(\theta^1,u^1)$, the orange lines to the initial values for state 2, $(\theta^2,u^2)$.  
                                         The green lines represent the analytical solution $\theta^1=\theta^2$ and $u^1=u^2$ for $t\in [0,T]$.  
                                         }
        \label{AA1}
    \end{figure}
                \begin{figure}
        \centering
                        \begin{subfigure}{0.5\textwidth}
            \centering
                \includegraphics[width=1\linewidth]{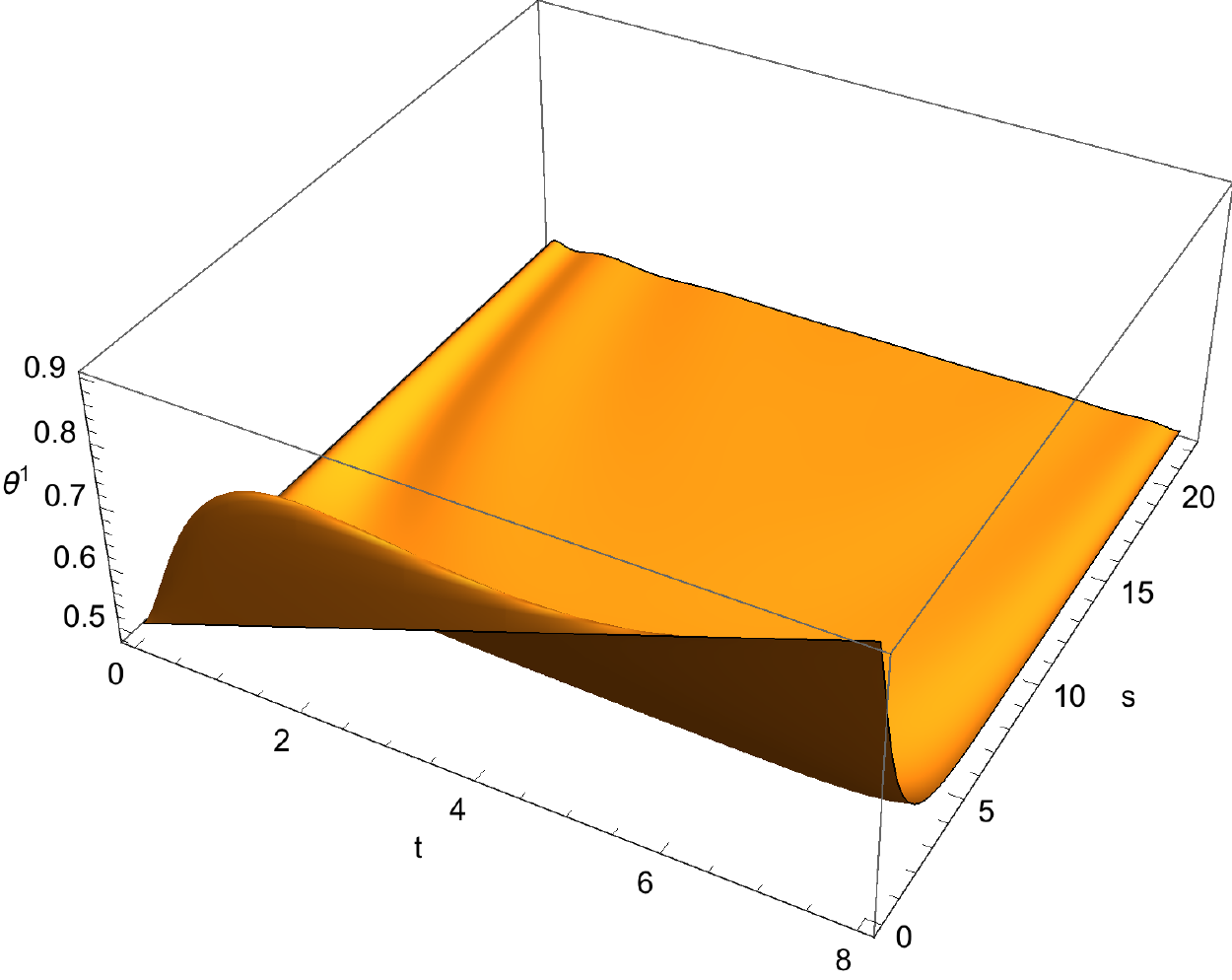}
                \caption{Distribution of players $\theta^1(t,s)$.}
               \label{fig:AA2:theta}
                        \end{subfigure}
            \begin{subfigure}{0.5\textwidth}
            \centering
                \includegraphics[width=1\linewidth]{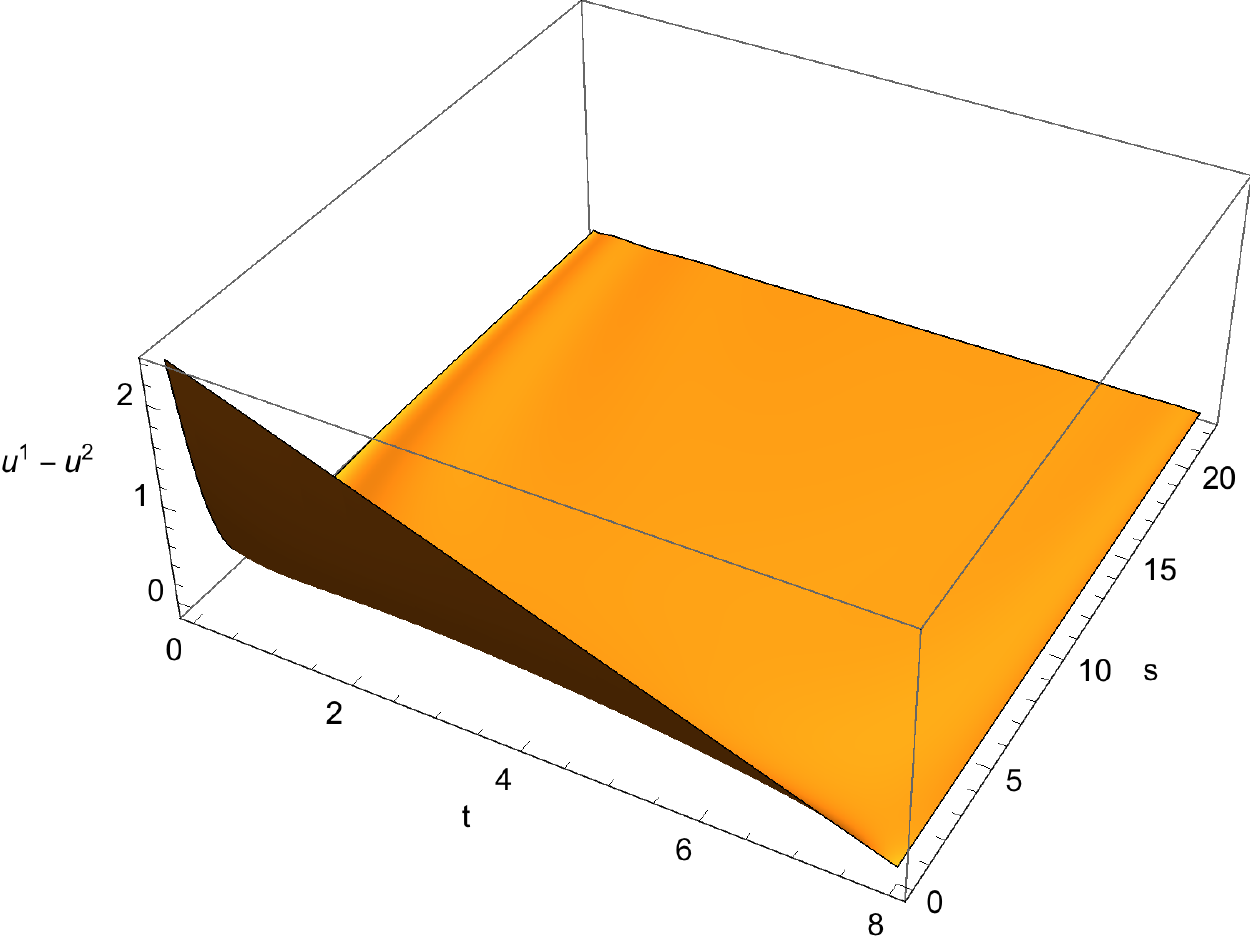}
                \caption{Value functions $u(t,s)$.}
            \end{subfigure}
        \caption{Evolution of $(u^1-u^2)(\cdot,s)$ and $\theta(\cdot,s)$, for $s\in[0,20]$.}
        \label{AA2}
        \end{figure}
                \begin{figure}
        \centering
                        \begin{subfigure}{0.5\textwidth}
            \centering
                \includegraphics[width=1\linewidth]{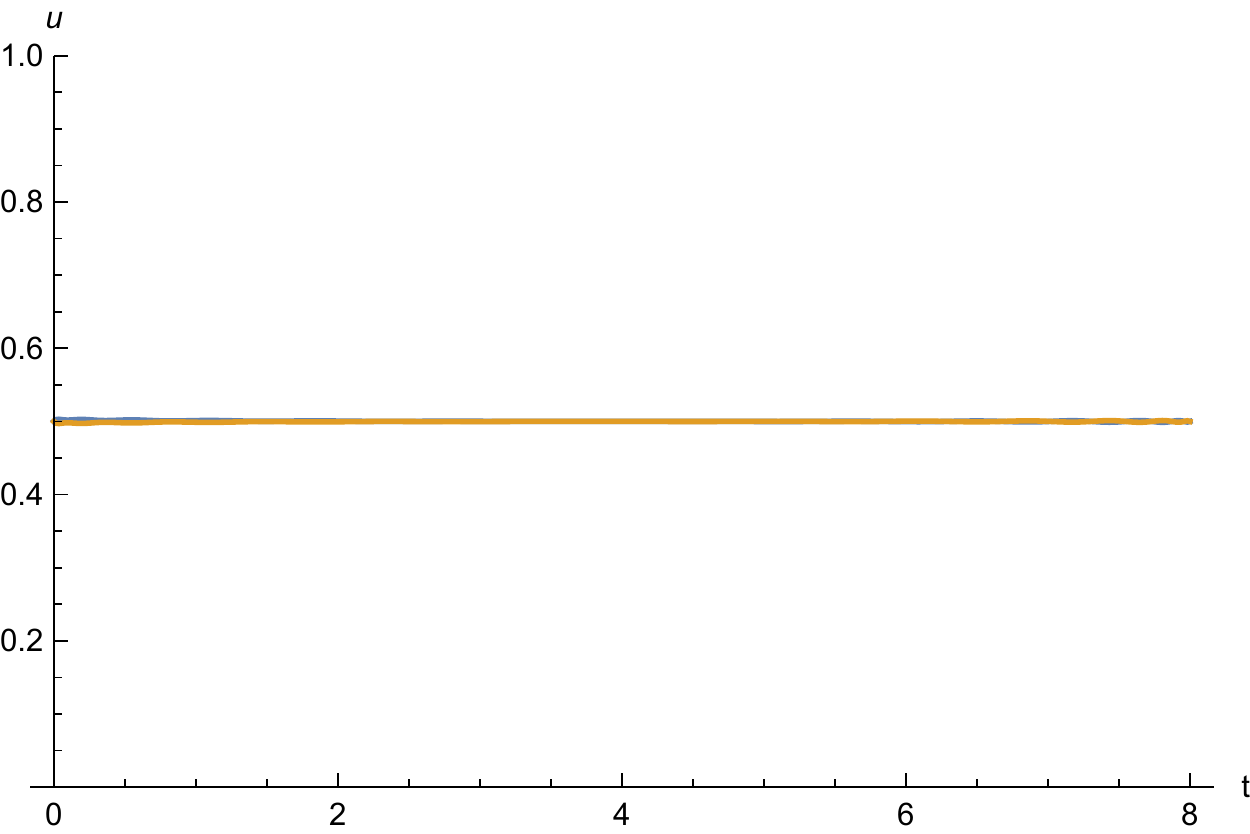}
                \caption{Final distribution $\theta(\cdot,s=100)$.}
                        \end{subfigure}
            \begin{subfigure}{0.5\textwidth}
            \centering
                \includegraphics[width=1\linewidth]{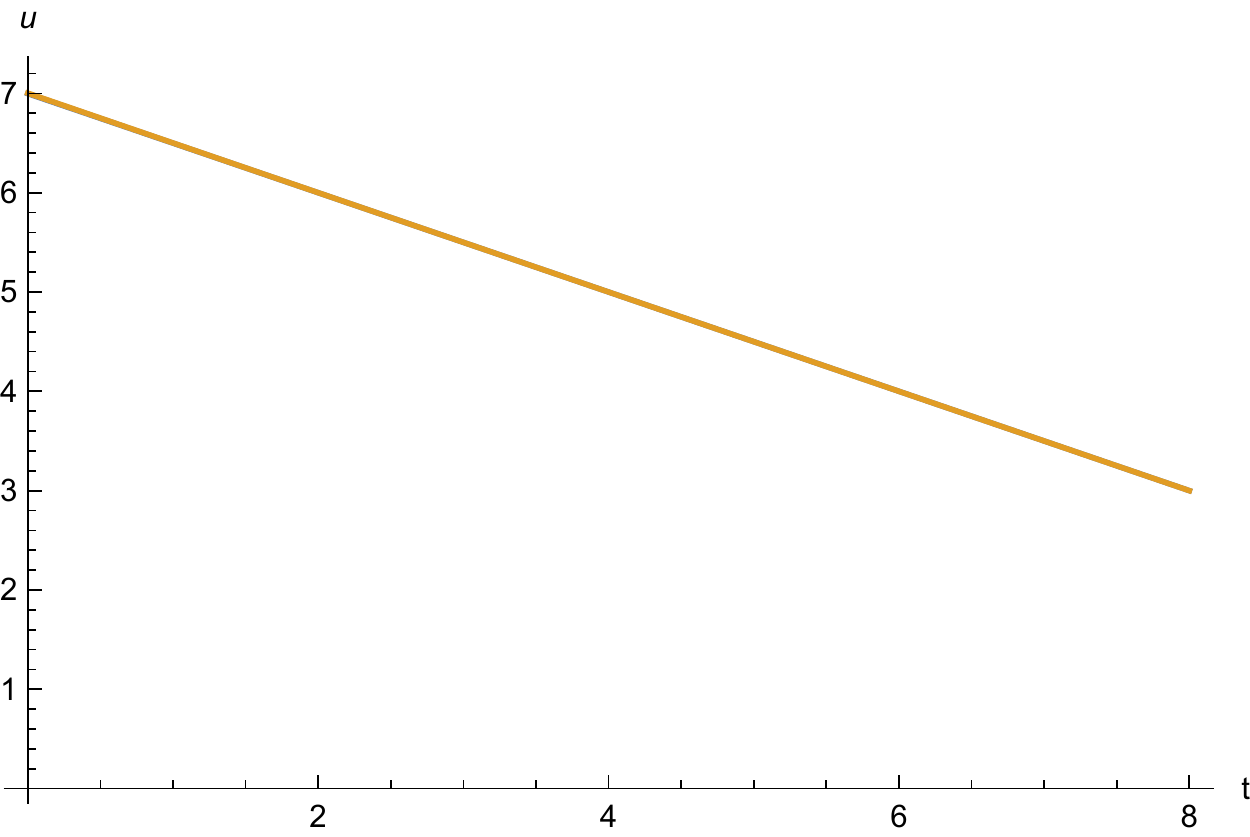}
                \caption{Final value function $u(\cdot,s=100)$.}
            \end{subfigure}
        \caption{Final value of $u(\cdot,s)$ and $\theta(\cdot,s)$, for $s=100$.}
        \label{AA3}
        \end{figure}
                \begin{figure}
                        \centering
                        \includegraphics[width=1\linewidth]{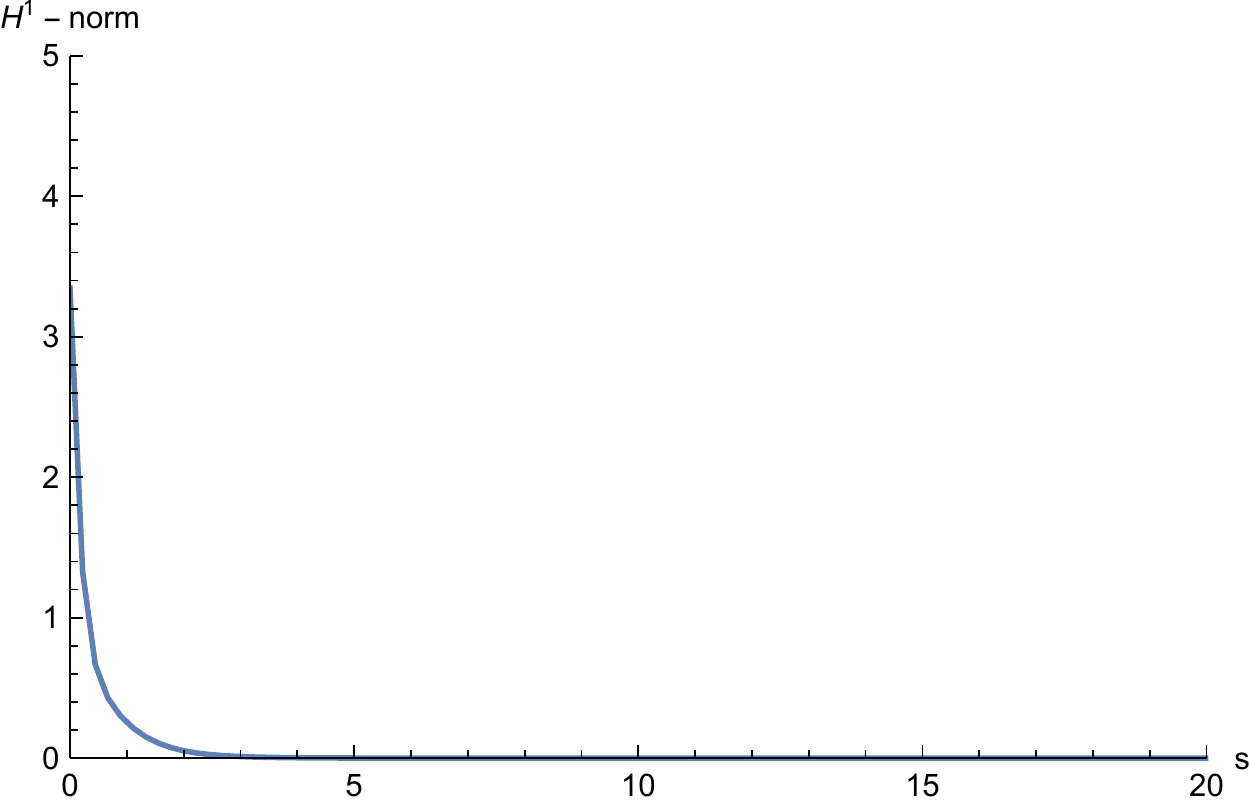}
                    \caption{Evolution, with the parameter $s$, of the $H^1$-norm of the difference between the computed solution $(u,\theta)(\cdot,s)$ and the analytical solution for the unconstrained probability case.}
                        \label{AA4}
                \end{figure}
                The paradigm-shift problem is a potential MFG with the Hamiltonian
                corresponding to
                \begin{equation*}
                \tilde h(\Delta_i u,i) = -\frac 1 2 ((u^i-u^j)^+)^2, \text{ and } F(\theta) = \frac{\theta_1^2 + \theta_2^2}{2}
                \end{equation*}
                in \eqref{hfor}.
                Thus, as a final test to our numerical method, we investigate the evolution of the Hamiltonian. 
                In this case, as expected, the Hamiltonian converges to a constant,  
                see Figure \ref{fig:cons_ham}.  
                \begin{figure} 
                        \centering
                        \includegraphics[width=1\linewidth]{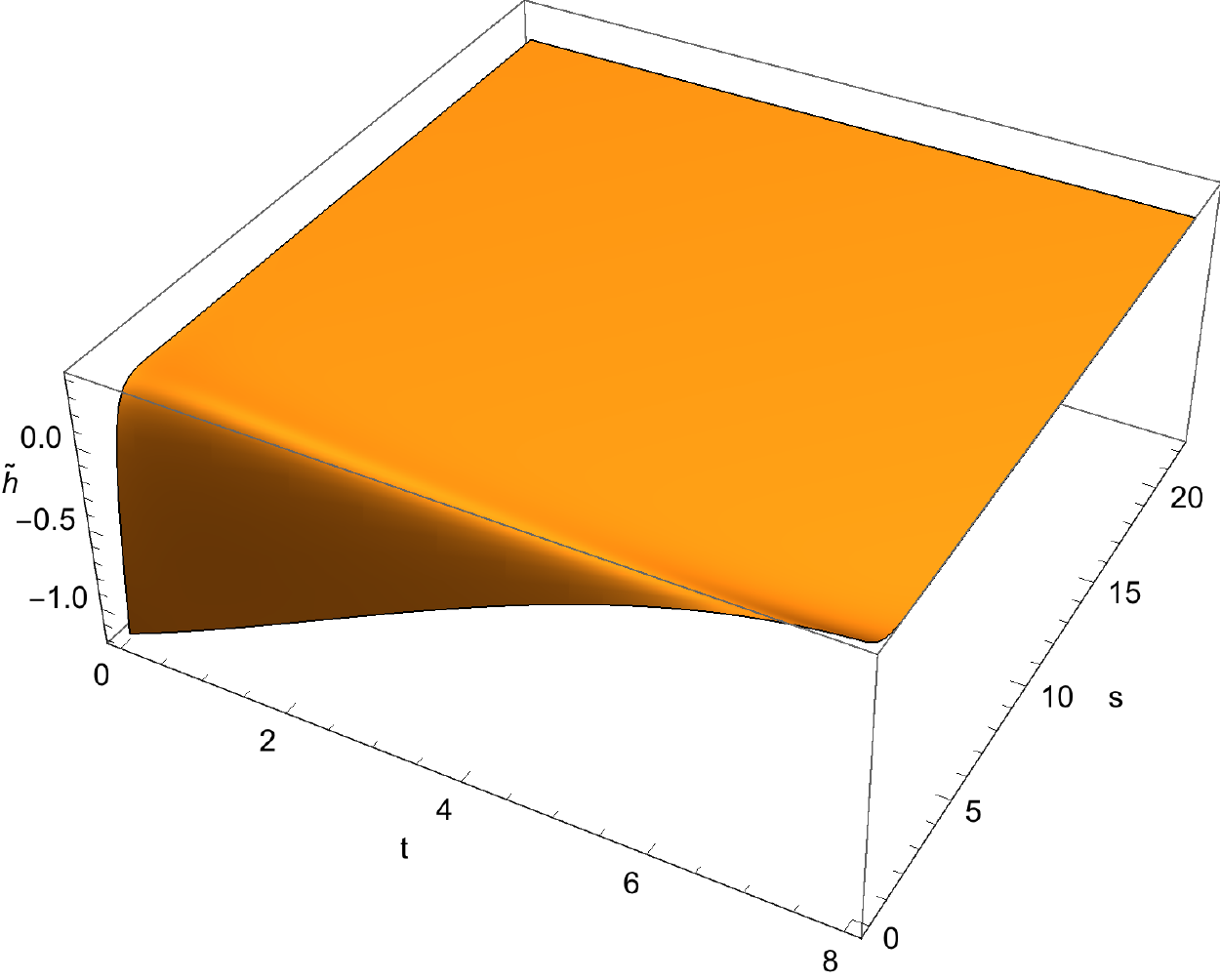}
                        \caption{Evolution of the Hamiltonian for $0\leq s\leq 20$.
                                }
                        \label{fig:cons_ham}
                \end{figure}

        In the preceding example, while iterating \eqref{eq:prj},  $\theta$ remains away from $0$. 
        In the next example, we consider a problem where, without the projection $P$ in \eqref{eq:prj}, positivity is not preserved.
        We set  $N=500$ and chose initial conditions as in Figure \ref{it_cond_proj}. 
    \begin{figure}
      \centering
      \begin{subfigure}{0.5\textwidth}
          \centering
              \includegraphics[width=1\linewidth]{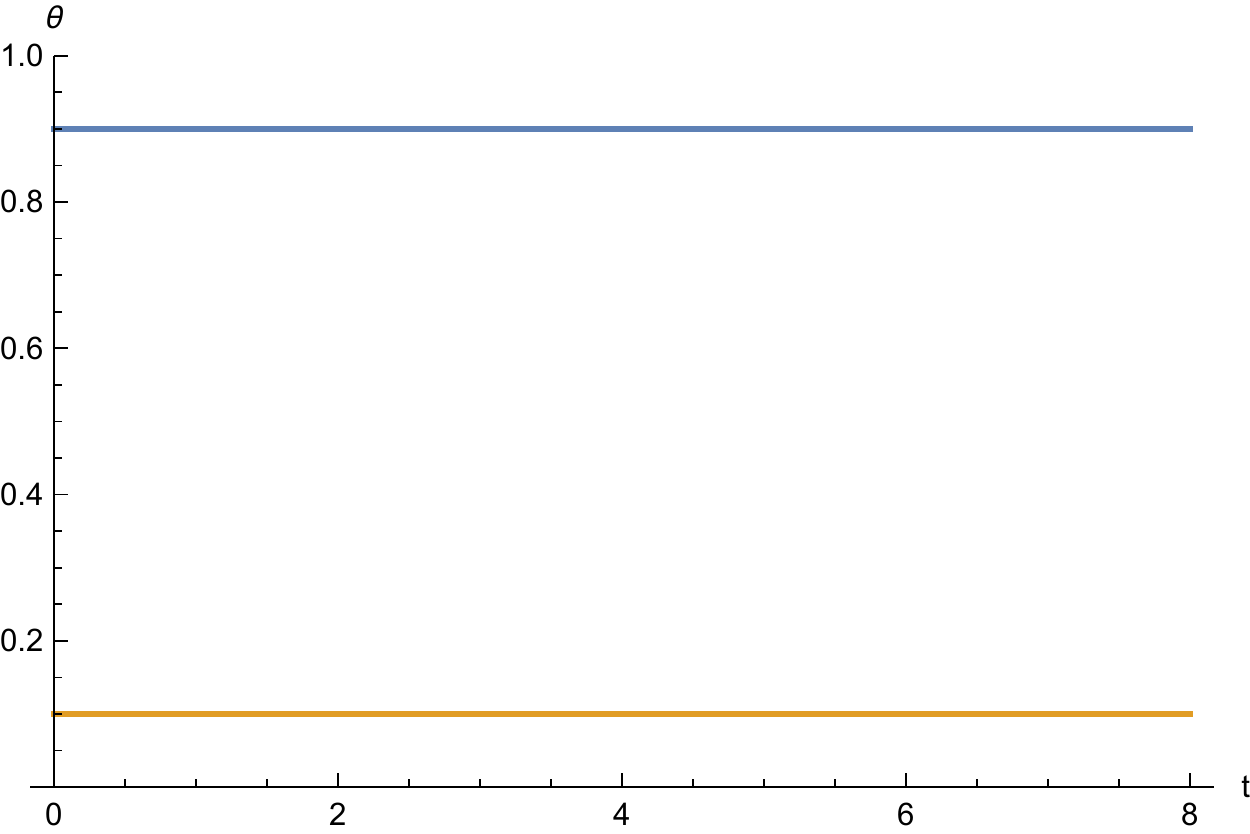}
              \caption{Initial condition $\theta(\cdot,s=0)$.}
             \label{fig:it_cond}
      \end{subfigure}
      \begin{subfigure}{0.5\textwidth}
      \centering
        \includegraphics[width=1\linewidth]{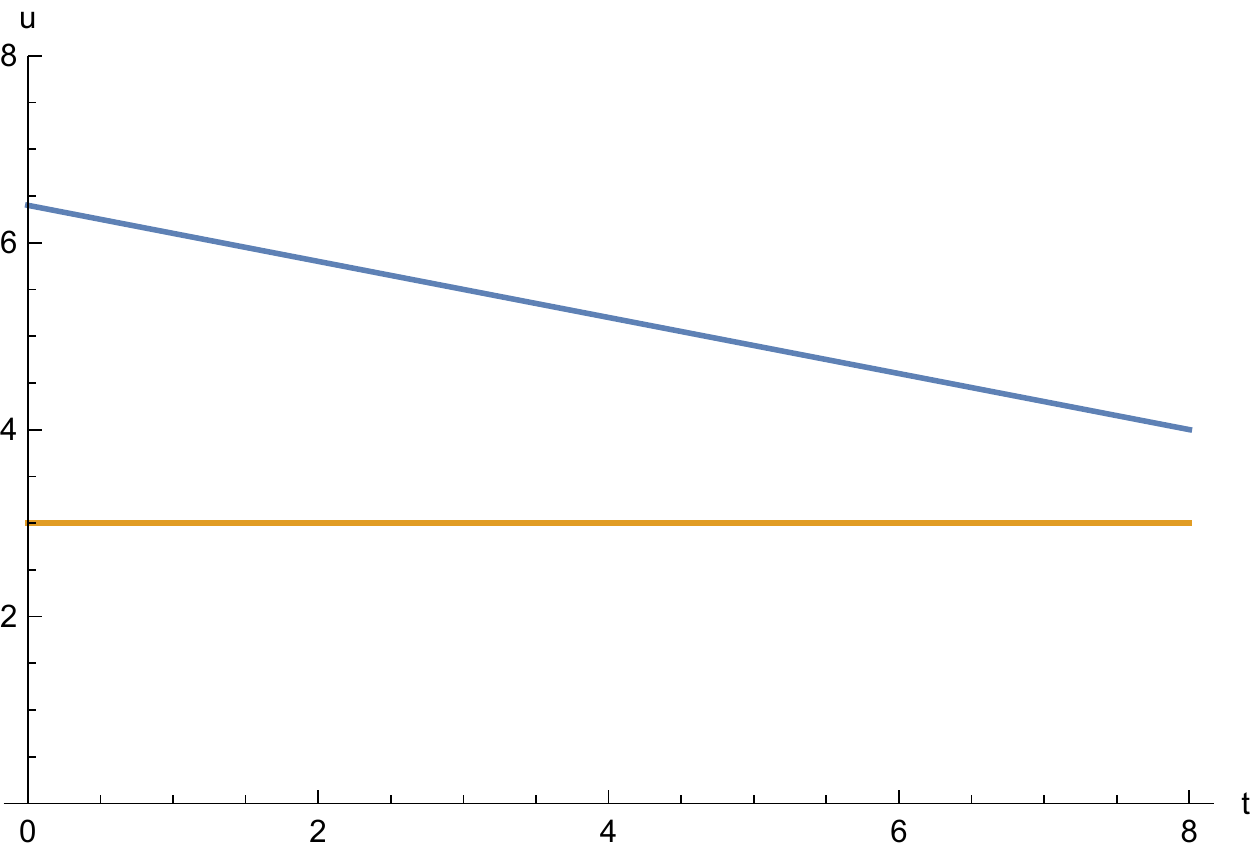}
        \caption{Initial condition $u(\cdot,s=0)$.}
      \end{subfigure}
      \caption{The blue lines correspond to the initial values ($s=0$) for state 1, $(\theta^1,u^1)$, the orange lines to the initial values for state 2, $(\theta^2,u^2)$.}
      \label{it_cond_proj}
    \end{figure}
        In Figure \ref{BB2}, we show the evolution by \eqref{eq:prj} for $s\in[0,20]$.
        In Figure \ref{BB3}, we see the final result for $s=100$. 
    Finally, in Figure \ref{BB4}, we show the evolution of the $H^1$ norm of the difference $\|(\tilde u, \tilde \theta) - (u,\theta) \|_{H^1([0,T])}^2(s)$, for $s\in[0,50]$.
              
        \begin{figure}
      \centering
                      \begin{subfigure}{0.5\textwidth}
          \centering
              \includegraphics[width=1\linewidth]{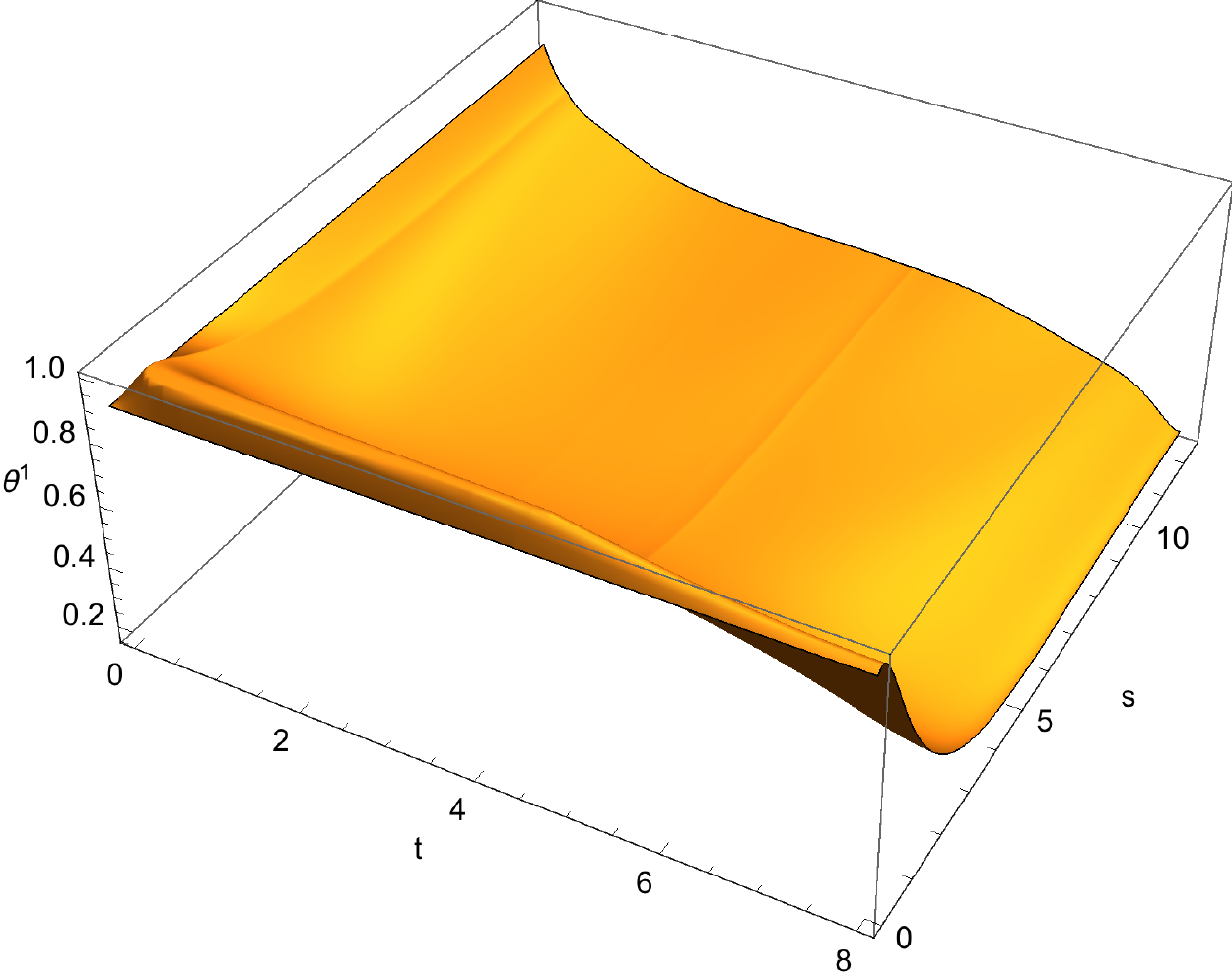}
              \caption{Distribution of players $\theta^1$, at each time $t$.}
                      \end{subfigure}
          \begin{subfigure}{0.5\textwidth}
          \centering
              \includegraphics[width=1\linewidth]{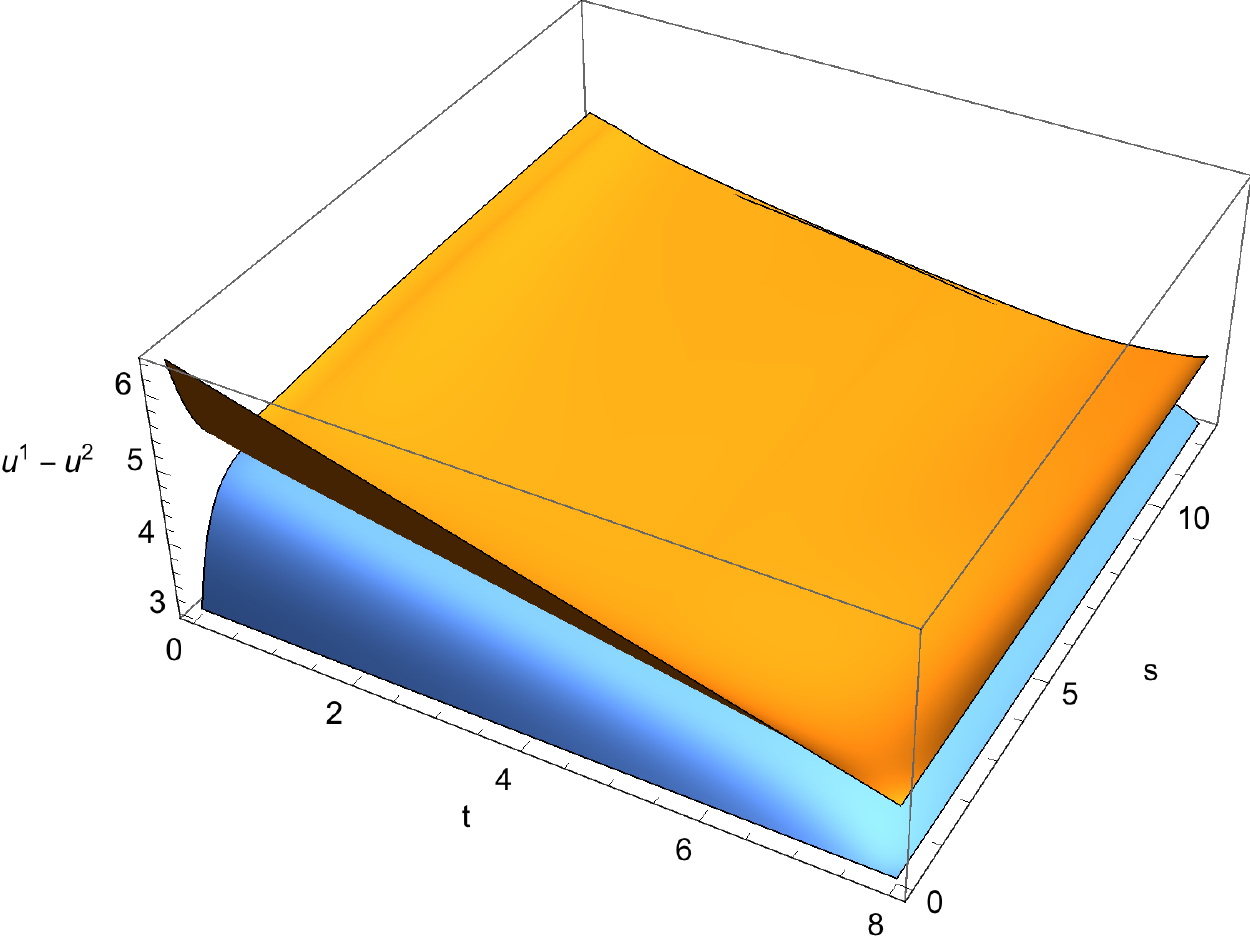}
              \caption{Value functions $u^1$, and $u^2$.}
          \end{subfigure}
      \caption{Evolution of $u(\cdot,s)$ and $\theta(\cdot,s)$, for $s\in[0,20]$. The quantities for the state 1 are depicted in blue and for state 2 in orange.}
      \label{BB2}
      \end{figure}
              \begin{figure}
      \centering
                      \begin{subfigure}{0.5\textwidth}
          \centering
              \includegraphics[width=1\linewidth]{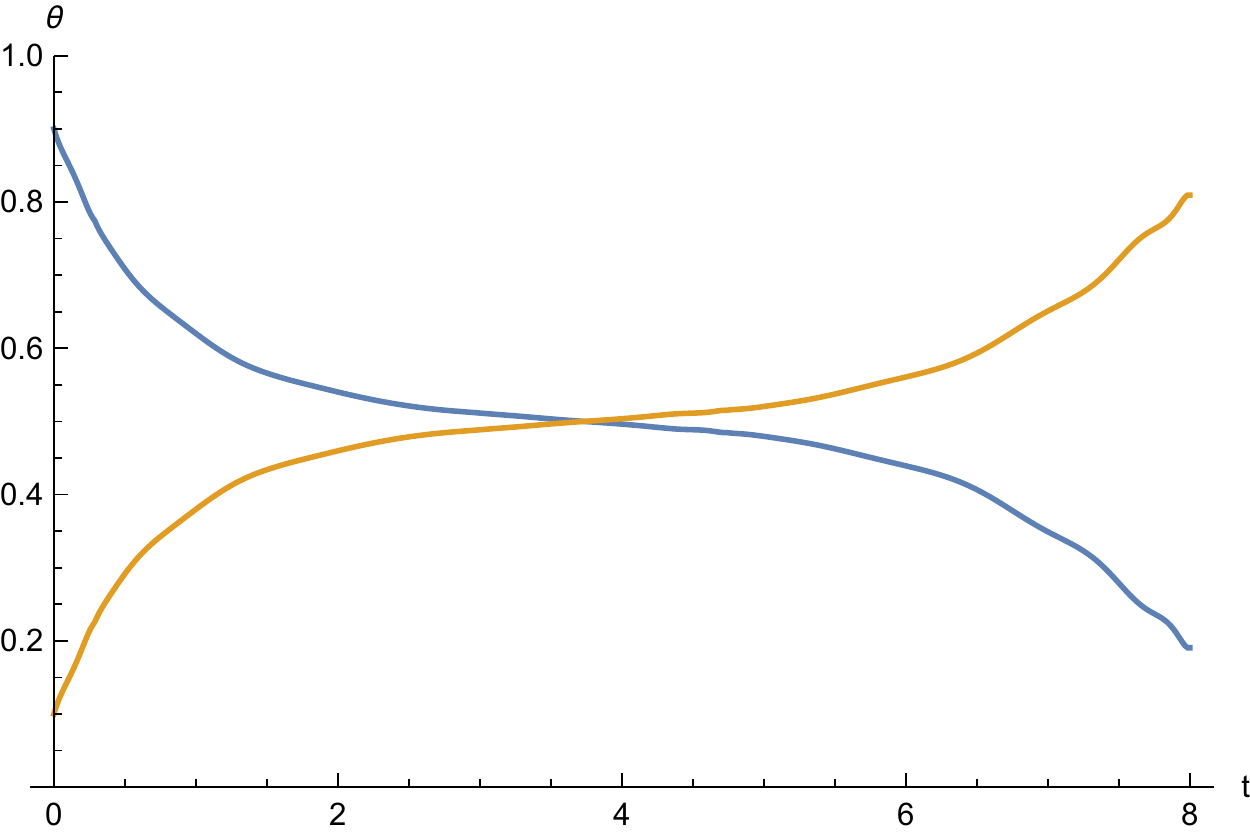}
              \caption{Final distribution of players $\theta(\cdot,100)$.}
                      \end{subfigure}
          \begin{subfigure}{0.5\textwidth}
          \centering
              \includegraphics[width=1\linewidth]{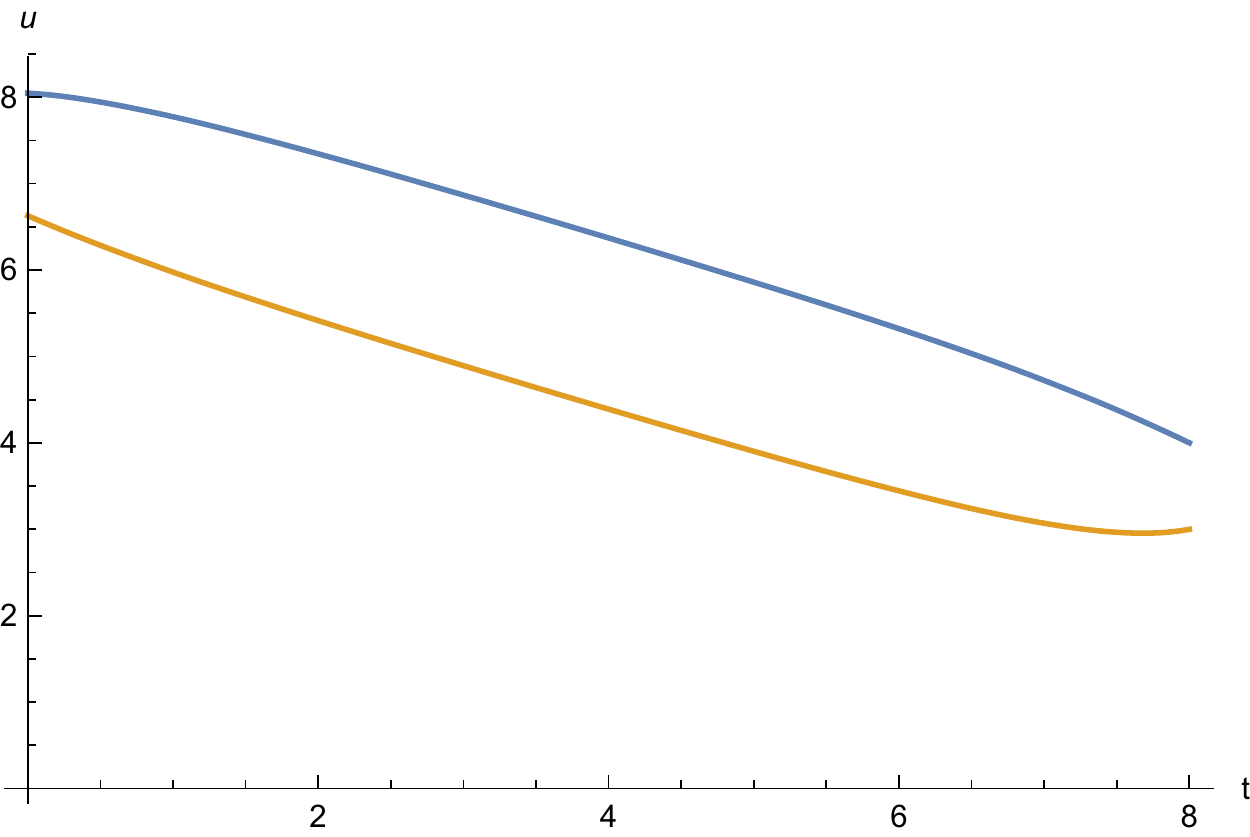}
              \caption{Final value functions $u(\cdot,100)$.}
          \end{subfigure}
      \caption{Final value of $u(\cdot,s)$ and distribution $\theta(\cdot,s)$, at $s=100$.}
      \label{BB3}
      \end{figure}
              \begin{figure}
                      \centering
                      \includegraphics[width=1\linewidth]{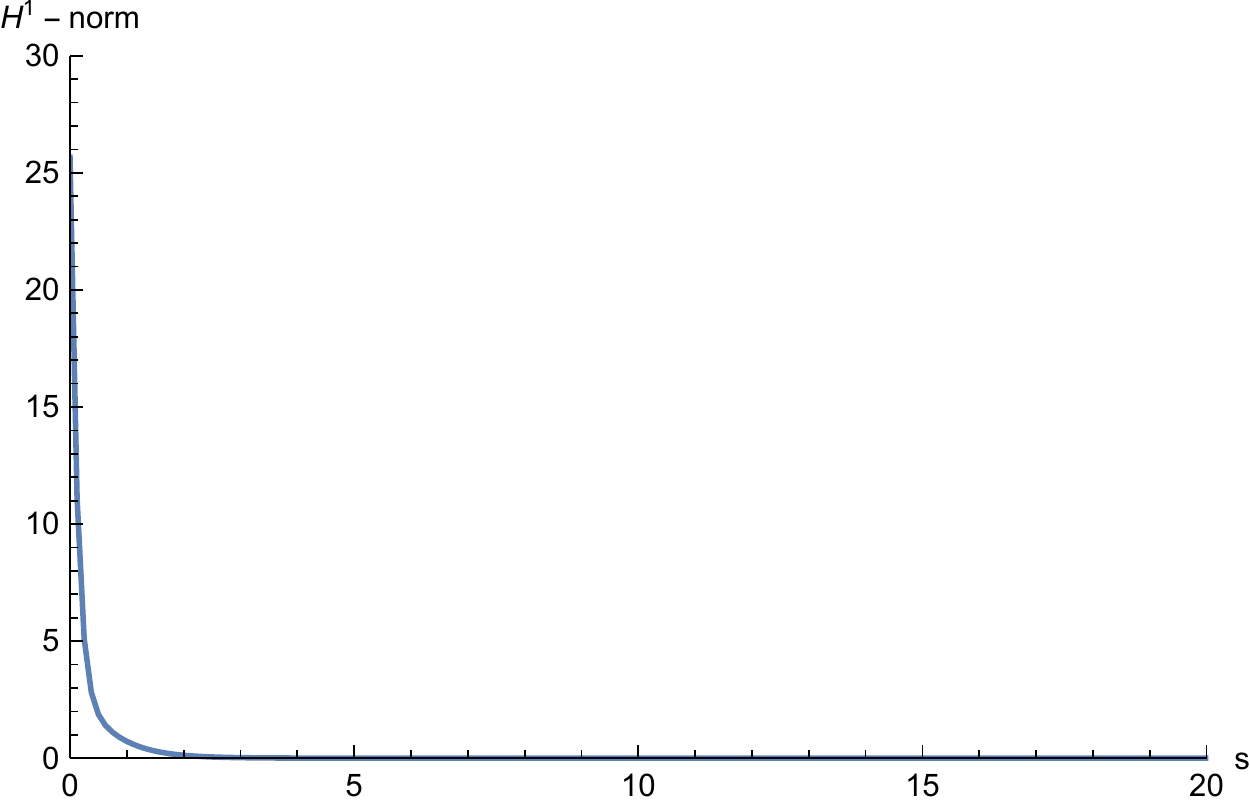}
                  \caption{Evolution, with respect to the parameter $s$, of the $H^1$-norm of the difference of the solution $(u,\theta)(\cdot,s)$ and the solution obtained at 
                                                   $s=100$: $\|(u,\theta)(\cdot,s) - (u,\theta)(\cdot,100)\|_{H^1}$. 
                                                   }
                      \label{BB4}
              \end{figure}
        In Figure \ref{fig:cons_ham_pos}, we plot the evolution of the Hamiltonian for using the projection method. Again, we obtain the numerical conservation of the Hamiltonian. 

        \begin{figure} 
                \centering
            \includegraphics[width=1\linewidth]{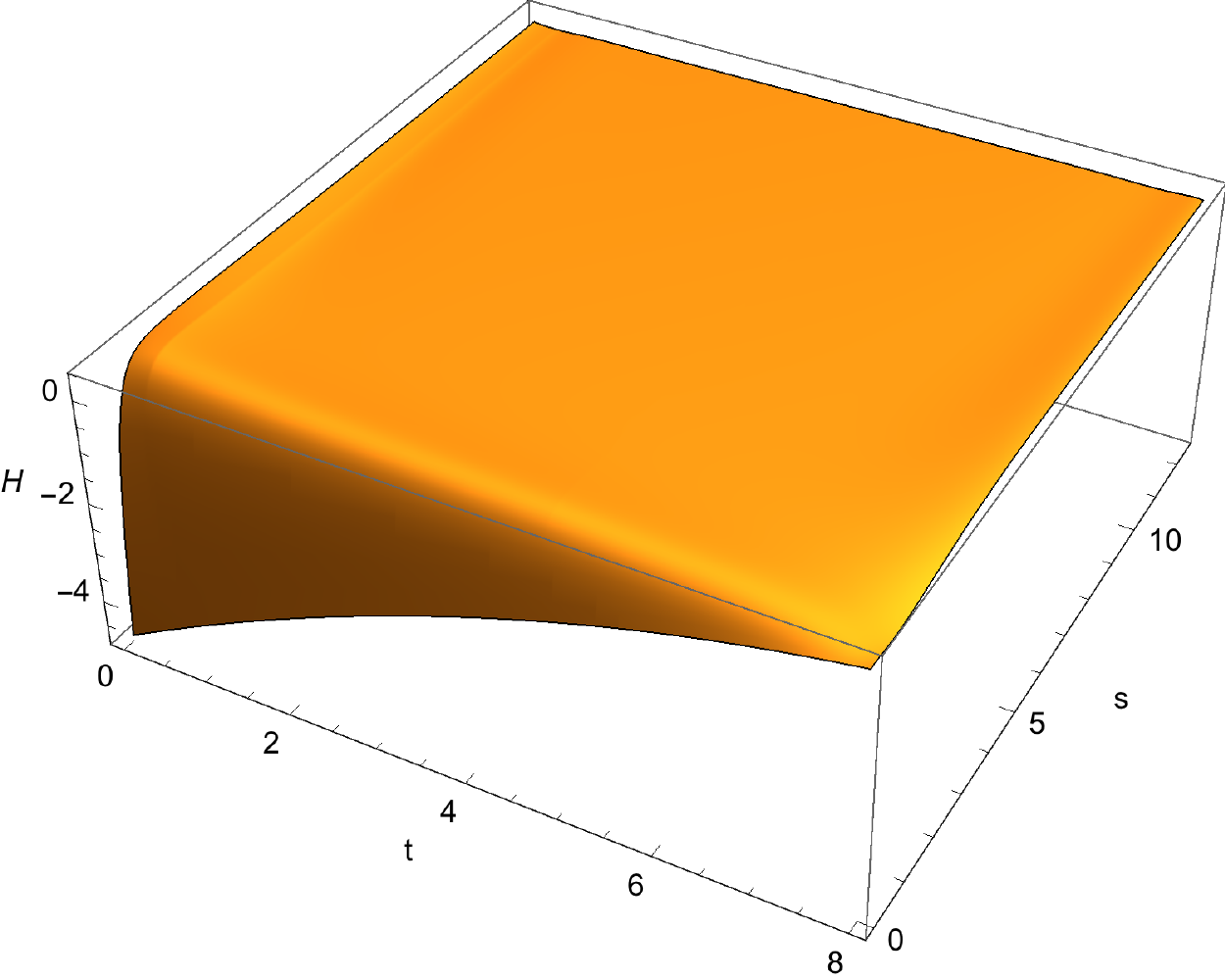}
            \caption{Evolution of the Hamiltonian with the $s$-dynamics that preserves the probability and the positivity of the distribution of players. 
                                        }
            \label{fig:cons_ham_pos}
        \end{figure}
\section{Conclusions} 
\label{sec:conclusions}
 
As the examples in the preceding sections illustrate, we have developed an effective method for the numerical approximation of monotonic finite-state MFGs.
As observed previously, \cite{MR2928376,achdou2013finite, DY, CDY}, monotonicity properties
are essential for the construction of effective numerical methods
and were used explicitly in \cite{AFG}. 
 Here, in contrast with earlier approaches, we
do not use a Newton-type iteration as in \cite{achdou2013finite, DY} nor require the solution of the master equation as in \cite{Gomes:2014kq}, which is numerically prohibitive for a large number of states. 
The key contribution of this work is the projection method developed in the previous section 
that made it possible to address the initial-terminal value problem. This was an open 
problem since the introduction of monotonicity-based methods in \cite{AFG}. 
 Our methods can be applied to discretize continuous-state MFGs, and we foresee additional extensions. The first one 
concerns the planning problem considered in \cite{CDY}. A second extension regards boundary 
value problems, which are natural in many applications of MFGs. Finally, our methods may be improved by using higher-order integrators in time, provided monotonicity is preserved. These matters will be
the subject of further research.

\end{document}